\documentclass[11pt,fleqn,a4paper]{article} 

\usepackage{amsmath,amssymb,amsthm,amsfonts,tikz-cd} 
\usepackage[mathscr]{eucal}

\usepackage{hyperref}
\hypersetup{colorlinks, linkcolor=blue, citecolor=blue, urlcolor=blue}

\allowdisplaybreaks

\newcommand{\p}{\partial}
\newcommand{\sgn}{\mathop{\rm sgn}\nolimits}
\newcommand{\ord}{\mathop{\rm ord}\nolimits}
\newcommand{\const}{{\rm const}}

\newtheorem{theorem}{Theorem}
\newtheorem{lemma}[theorem]{Lemma}
\newtheorem{corollary}[theorem]{Corollary}
\newtheorem{proposition}[theorem]{Proposition}
\newtheorem{conjecture}[theorem]{Conjecture}
{\theoremstyle{definition}
	
	\newtheorem{remark}[theorem]{Remark}
}

\newcommand{\todo}[1][\null]{\ensuremath{\clubsuit}}

\newcommand{\noprint}[1]{}
\newcommand{\lsemioplus}{\mathbin{\mbox{$\lefteqn{\hspace{.77ex}\rule{.4pt}{1.2ex}}{\in}$}}}

\begin{document}

\par\noindent {\LARGE\bf
Point and generalized symmetries of the heat equation revisited
\par}

\vspace{4mm}\par\noindent{\large
\large Serhii D. Koval$^{\dag}$ and Roman O. Popovych$^{\ddag}$
}
	
\vspace{4mm}\par\noindent{\it\small
$^\dag$Department of Mathematics and Statistics, Memorial University of Newfoundland,\\
$\phantom{^\ddag}$\,St.\ John's (NL) A1C 5S7, Canada\\
$\phantom{^\ddag}$\,Department of Mathematics, Kyiv Academic University, 36 Vernads'koho Blvd, 03142 Kyiv, Ukraine
\par}

\vspace{2mm}\par\noindent{\it\small
$^\ddag$\,Mathematical Institute, Silesian University in Opava, Na Rybn\'\i{}\v{c}ku 1, 746 01 Opava, Czech Republic\\
$\phantom{^\S}$Institute of Mathematics of NAS of Ukraine, 3 Tereshchenkivska Str., 01024 Kyiv, Ukraine
\par}

\vspace{4mm}\par\noindent
E-mails:
skoval@mun.ca, rop@imath.kiev.ua
	
\vspace{5mm}\par\noindent\hspace*{10mm}\parbox{140mm}{\small
We derive a nice representation for point symmetry transformations
of the (1+1)-dimensional linear heat equation and properly interpret them.
This allows us to prove that the pseudogroup of these transformations has exactly two connected components.
That is, the heat equation admits a single independent discrete symmetry,
which can be chosen to be alternating the sign of the dependent variable.
We introduce the notion of pseudo-discrete elements of a Lie group
and show that alternating the sign of the space variable,
which was for a long time misinterpreted as a discrete symmetry of the heat equation,
is in fact a pseudo-discrete element of its essential point symmetry group.
The classification of subalgebras of the essential Lie invariance algebra of the heat equation is enhanced
and the description of generalized symmetries of this equation is refined as well.
We also consider the Burgers equation because of its relation to the heat equation
and prove that it admits no discrete point symmetries.
The developed approach to point-symmetry groups
whose elements have components that are linear fractional in some variables
can directly be extended to many other linear and nonlinear differential equations.
}\par\vspace{4mm}
	
\noprint{
Keywords:
heat equation;
point-symmetry pseudogroup;
Lie symmetry;
discrete symmetry;
subalgebra classification;
generalized symmetry

MSC: 35K05, 35B06, 35A30

35-XX Partial differential equations
  35Kxx Parabolic equations and parabolic systems {For global analysis, analysis on manifolds, see 58J35}
    35K05 Heat equation
    35K10 Second-order parabolic equations
  35Qxx	Partial differential equations of mathematical physics and other areas of application [See also 35J05, 35J10, 35K05, 35L05]
    35Q79  	PDEs in connection with classical thermodynamics and heat transfer
  35Axx General topics
    35A30 Geometric theory, characteristics, transformations [See also 58J70, 58J72]
  35Bxx Qualitative properties of solutions
    35B06 Symmetries, invariants, etc.
}

\section{Introduction}

The (1+1)-dimensional (linear) heat equation
\begin{gather}\label{eq:LinHeat}
u_t=u_{xx}
\end{gather}
is one of the simplest but most fundamental equations of mathematical physics.
This equation became a test example in a number of branches within the theory of differential equations,
including symmetry analysis of such equations.
Studying symmetries and related objects of the equation~\eqref{eq:LinHeat}
was initiated by Sophus Lie himself~\cite{lie1881a}
in the course of group classification of second-order linear partial differential equations in two independent variables.
In particular, he computed its maximal Lie invariance algebra and
showed that it gives a unique (modulo the point equivalence) maximal Lie-symmetry extension
in the class of linear (1+1)-dimensional second-order evolution equations.
In the present, the heat equation is the first standard equation for testing
packages for symbolic computation of symmetries of various kinds and related objects
for differential equations.
It was the equation~\eqref{eq:LinHeat} that was used as the only example
for introducing the concept of nonclassical reduction in~\cite{blum1969a}.
Such reductions of~\eqref{eq:LinHeat} were first completely described only in~\cite{fush1992e}
(see also~\cite{webb1990a} for a preliminary study),
and this result originated studying singular reduction modules and no-go problems on nonclassical reductions
for general partial differential equations \cite{boyk2016a,kunz2008b,popo1998b,popo2008b,zhda1998a}.
The space of local conservation laws of the equation~\eqref{eq:LinHeat}
is known for a long time~\cite{Dorodnitsyn&Svirshchevskii1983};
more specifically, the space of its reduced conservation-law characteristics
coincides with the solution space of the backward heat equation.
The theorem \cite[Theorem~8]{popo2005b} that
any potential conservation law of the equation~\eqref{eq:LinHeat} is equivalent,
on the solution set of the corresponding potential system, to a local conservation law of this equation
was generalized in~\cite{popo2008a} to an arbitrary linear (1+1)-dimensional second-order evolution equation.
Therein, potential symmetries of such equations, including equations~\eqref{eq:LinHeat},
and Darboux transformations between them were comprehensively studied following~\cite{matv1991A}.
To gain an impression of the state of the art in symmetry analysis of the equation~\eqref{eq:LinHeat},
see, e.g., Examples~2.41, 3.3, 3.13, 3.17 and~5.21 in \cite{olve1993A}, \cite[Section~10.1]{CRChandbook1994V1},
\cite{fush1992e,opan2022a,popo2008a}, \cite[Section~A]{vane2021} and \cite[p.~531--535]{wint1990a}.

In spite of the rich and diverse history of studying the equation~\eqref{eq:LinHeat},
a number of basic problems related to it even within the framework of classical group analysis
still require refinement.
Thus, a neat description of the point symmetry pseudogroup~$G$ of this equation
and an accurate classification of subalgebras of its essential Lie invariance algebra~$\mathfrak g^{\rm ess}$
have not yet been presented in the literature.
Improper interpretations of continuous and discrete symmetries of the equation~\eqref{eq:LinHeat}
led to the inconsistency between the action of the essential point symmetry group~$G^{\rm ess}$ on~$\mathfrak g^{\rm ess}$
and the inner automorphism group of~$\mathfrak g^{\rm ess}$.
In the present work, we successfully solve the above problems using an approach from~\cite{kova2022a}
for enhancing the representation of point symmetries of the equation~\eqref{eq:LinHeat}
given in the proof of Theorem~8 in~\cite{opan2022a}.
After achieving the consistency between the action of~$G^{\rm ess}$ on~$\mathfrak g^{\rm ess}$
and the inner automorphism group of~$\mathfrak g^{\rm ess}$,
we first construct an accurate optimal list of subalgebras of~$\mathfrak g^{\rm ess}$ and
an optimal list of one-dimensional subalgebras of~$\mathfrak g$.
We also introduce the notion of pseudo-discrete elements of a Lie group.
It turns out that the point transformation only alternating the sign of the space variable~$x$,
which was incorrectly assumed to be a discrete point symmetry transformation of the equation~\eqref{eq:LinHeat},
is in fact a pseudo-discrete element of~$G^{\rm ess}$.

It is surprising that the explicit description of generalized symmetries of the equation~\eqref{eq:LinHeat}
is not well known and was presented only in \cite[Section~10.1]{CRChandbook1994V1} without proof
although it can be straightforwardly derived from well-known particular results on these symmetries
from, e.g., \cite[Example~5.21]{olve1993A} and \cite[Section~4.4.2]{boch1999A}.
At the same time, this description can also be easily obtained from the very beginning
using the Shapovalov--Shirokov theorem~\cite[Theorem~4.1]{shap1992a}.
The algebra~$\Sigma$ of generalized symmetries of the equation~\eqref{eq:LinHeat}
is the semidirect sum of the subalgebra~$\Lambda$ of linear generalized symmetries of~\eqref{eq:LinHeat}
and the ideal~$\Sigma^{-\infty}$ associated with the linear superposition of solutions of~\eqref{eq:LinHeat},
and we prove that the subalgebra~$\Lambda$ is generated from the simplest nonzero linear generalized symmetry $u\p_u$
by two recursion operators, which are associated with Lie symmetries of space translations and Galilean boosts.
Hence this subalgebra is isomorphic to the Lie algebra \smash{${\rm W}(1,\mathbb R)^{(-)}$}
associated with the rank-one Weyl algebra ${\rm W}(1,\mathbb R)$.

Since the Burgers equation is related to the equation~\eqref{eq:LinHeat}
via linearizing by the Hopf--Cole transformation
and the $t$-components of its point symmetry transformations are linear fractional in~$t$,
in Section~\ref{sec:LinHeatBurgers} we extend the suggested approach to the Burgers equation.
We enhance the representation for its point symmetries given in~\cite{poch2013a,poch2017a},
properly interpret them and show that the Burgers equation admits no discrete point symmetries,
whereas alternating the sign of the space variable is its pseudo-discrete point symmetry.

The projections of the point-symmetry (pseudo)groups of
many linear and nonlinear systems of differential equations, including multi-dimensional ones,
to subspaces coordinatized by certain system variables
consist of linear fractional transformations,
see the beginning of Section~\ref{sec:LinHeatBurgers}.
The consideration of the remarkable (1+2)-dimensional Fokker--Planck equation
(also known as the Kolmogorov equation) $u_t+xu_y=u_{xx}$ in~\cite{kova2022a}
and of the heat and Burgers equations in the present paper clearly shows
that the developed approach to interpreting transformations with linear fractional components
can directly be extended to all such systems.

The structure of the paper is as follows.
In Section~\ref{sec:LinHeatMIA}, we present the maximal Lie invariance algebra of the equation~\eqref{eq:LinHeat}
and describe its key properties.
Using the direct method, in Section~\ref{sec:LinHeatPointSymGroup}
we re-compute the point symmetry pseudogroup of this equation
and analyze its structure, including its decomposition and the description of its discrete elements.
Section~\ref{sec:LinHeatClassIneqSubalgebras} is devoted to the classifications
of subalgebras of~$\mathfrak g^{\rm ess}$ of all possible dimensions up to the $G^{\rm ess}$-equivalence and
of one-dimensional subalgebras of~$\mathfrak g$ modulo the $G$-equivalence.
In Section~\ref{sec:LinHeatPseudo-discreteSyms}, we solve
the problem of characterizing elements of~$G^{\rm ess}$ belonging to~$\exp(\mathfrak g^{\rm ess})$.
We also introduce the notion of pseudo-discrete element of a Lie group and show that
the point transformation only alternating the sign of~$x$ is a pseudo-discrete element of~$G^{\rm ess}$.
The description of generalized symmetries of the equation~\eqref{eq:LinHeat}
is essentially refined and completed in Section~\ref{sec:LinHeatGenSyms}.
The developed approach to point symmetry (pseudo)groups with linear fractional transformation components
is applied to the Burgers equation in Section~\ref{sec:LinHeatBurgers}.
The results of the paper are analyzed in Section~\ref{sec:Conclusion}.

\section{Lie invariance algebra}\label{sec:LinHeatMIA}

The maximal Lie invariance algebra~$\mathfrak g$ of the equation~\eqref{eq:LinHeat} is spanned by the vector fields
\begin{gather*}
\mathcal P^t =\p_t, \quad
\mathcal D   =2t\p_t+x\p_x-\tfrac12u\p_u,\quad
\mathcal K   =t^2\p_t+tx\p_x-\tfrac14(x^2+2t)u\p_u,\\
\mathcal G^x =t\p_x-\tfrac12xu\p_u,\quad
\mathcal P^x =\p_x,\quad
\mathcal I   =u\p_u,\quad
\mathcal Z(f)=f(t,x)\p_u,
\end{gather*}
where the parameter function $f$ depends on~$(t,x)$ and runs through the solution set of the equation~\eqref{eq:LinHeat}.
The contact invariance algebra~$\mathfrak g_{\rm c}$ of the equation~\eqref{eq:LinHeat}
is just the first prolongation of the algebra~$\mathfrak g$, $\mathfrak g_{\rm c}=\mathfrak g_{(1)}$.

The vector fields $\mathcal Z(f)$ constitute the infinite-dimensional abelian ideal $\mathfrak g^{\rm lin}$ of~$\mathfrak g$
associated with the linear superposition of solutions of~\eqref{eq:LinHeat}, $\mathfrak g^{\rm lin}:=\{\mathcal Z(f)\}$.
Thus, the algebra $\mathfrak g$ can be represented as a semidirect sum, $\mathfrak g=\mathfrak g^{\rm ess}\lsemioplus\mathfrak g^{\rm lin}$,
where
\begin{gather}\label{eq:LinHeatEssAlg}
\mathfrak g^{\rm ess}=\langle\mathcal P^t,\mathcal D,\mathcal K,\mathcal G^x,\mathcal P^x,\mathcal I\rangle
\end{gather}
is a (six-dimensional) subalgebra of $\mathfrak g$,
called the \emph{essential Lie invariance algebra} of~\eqref{eq:LinHeat}.

Up to the antisymmetry of the Lie bracket of vector fields,
the nonzero commutation relations between the basis elements of~$\mathfrak g^{\rm ess}$ are exhausted by
\begin{gather*}
[\mathcal D,\mathcal P^t]  =-2\mathcal P^t,\quad
[\mathcal D,  \mathcal K]  =2\mathcal K,\quad
[\mathcal P^t,\mathcal K]  =\mathcal D,\\
[\mathcal P^t,\mathcal G^x]=\mathcal P^x,\quad
[\mathcal D,\mathcal G^x]  =\mathcal G^x,\quad
[\mathcal D,\mathcal P^x]  =-\mathcal P^x,\quad
[\mathcal K,\mathcal P^x]  =-\mathcal G^x,\\
[\mathcal G^x,\mathcal P^x]=\tfrac12\mathcal I.
\end{gather*}

The algebra $\mathfrak g^{\rm ess}$ is nonsolvable.
Its Levi decomposition is given by $\mathfrak g^{\rm ess}=\mathfrak f\lsemioplus\mathfrak r$,
where the radical~$\mathfrak r$ of~$\mathfrak g^{\rm ess}$ coincides with the nilradical of~$\mathfrak g^{\rm ess}$ and
is spanned by the vector fields $\mathcal G^x$, $\mathcal P^x$ and~$\mathcal I$.
The Levi factor $\mathfrak f=\langle\mathcal P^t,\mathcal D,\mathcal K\rangle$ of~$\mathfrak g^{\rm ess}$
is isomorphic to ${\rm sl}(2,\mathbb R)$,
the radical~$\mathfrak r$ of~$\mathfrak g^{\rm ess}$ is isomorphic to the (real) rank-one Heisenberg algebra ${\rm h}(1,\mathbb R)$,
and the real representation of the Levi factor~$\mathfrak f$ on the radical~$\mathfrak r$
coincides, in the basis $(\mathcal G^x,\mathcal P^x,\mathcal I)$,
with the representation $\rho_1\oplus \rho_0$ of~${\rm sl}(2,\mathbb R)$.
Here $\rho_n$ is the standard irreducible representation of~${\rm sl}(2,\mathbb R)$ on $\mathbb R^{n+1}$.
More specifically,
$\rho_n( \mathcal P^t)_{ij}=(n-j)\delta_{i,j+1}$,
$\rho_n( \mathcal D)_{ij}  =(n-2j)\delta_{ij}$,
$\rho_n(-\mathcal K)_{ij}  =j\delta_{i+1,j}$,
where $i,j\in\{0,\dots,n\}$, $n\in\mathbb N_0:=\mathbb N\cup\{0\}$,
and $\delta_{kl}$ is the Kronecker delta, i.e., $\delta_{kl}=1$ if $k=l$ and $\delta_{kl}=0$ otherwise, $k,l\in\mathbb N_0$.
Thus, the entire algebra~$\mathfrak g^{\rm ess}$ is isomorphic to
the so-called special Galilei algebra ${\rm sl}(2,\mathbb R)\lsemioplus_{\rho_1\oplus \rho_0}{\rm h}(1,\mathbb R)$,
which is denoted by~$L_{6,2}$
in the classification of indecomposable Lie algebras of dimensions up to eight with nontrivial Levi decompositions
from~\cite{turk1988a}.

The radical~$\mathfrak r$ and its derived algebra~$\mathfrak r':=[\mathfrak r,\mathfrak r]=\langle\mathcal I\rangle$,
which coincides with the center~$\mathfrak z(\mathfrak g^{\rm ess})$ of~$\mathfrak g^{\rm ess}$,
are the only proper megaideals and, moreover, the only proper ideals of~$\mathfrak g^{\rm ess}$.

Another basis of~$\mathfrak g^{\rm ess}$,
which stems from the Iwasawa decomposition of ${\rm SL}(2,\mathbb R)$
and is thus more convenient in many aspects, is
$(\mathcal Q^+,\mathcal D,\mathcal P^t,\mathcal G^x,\mathcal P^x,\mathcal I)$,
where $\mathcal Q^\pm:=\mathcal P^t\pm\mathcal K$.

\section{Complete point symmetry pseudogroup}\label{sec:LinHeatPointSymGroup}

The equation~\eqref{eq:LinHeat} belongs to the class $\mathcal E$
of linear (1+1)-dimensional second-order evolution equations of the general form
\begin{gather}\label{eq:LinHeatClassE}
u_t=A(t,x)u_{xx}+B(t,x)u_x+C(t,x)u+D(t,x)\quad \mbox{with}\quad A\ne0.
\end{gather}
Here the tuple of arbitrary elements of $\mathcal E$ is $\theta:=(A,B,C,D)\in\mathcal S_{\mathcal E}$,
and $S_{\mathcal E}$ is the solution set of the auxiliary system
consisting of the single inequality $A\ne0$
and the equations meaning that the arbitrary elements depend
at most on~$(t,x)$, $A_u=A_{u_x}=A_{u_t}=A_{u_{tt}}=A_{u_{tx}}=A_{u_{xx}}=0$
and similar equations for~$B$, $C$ and~$D$.

To find the point symmetry pseudogroup~$G$ of the equation~\eqref{eq:LinHeat},
we start with considering the equivalence groupoid of the class $\mathcal E$,
which in its turn is a natural choice for a (normalized) superclass for the equation~\eqref{eq:LinHeat}.
We use the papers~\cite{opan2022a,popo2008a} as reference points for known results on admissible transformations of the class $\mathcal E$.

\begin{proposition}[\cite{popo2008a}]\label{prop:EquivalenceGroupE}
The class~$\mathcal E$ is normalized in the usual sense.
Its usual equivalence pseudogroup~$G^\sim_{\mathcal E}$ consists of the transformations of the~form%
\begin{subequations}\label{eq:PointTransInEA}
\begin{gather}\label{eq:GenFormOfPointTransInEA}
\tilde t=T(t),\quad
\tilde x=X(t,x),\quad
\tilde u=U^1(t,x)u+U^0(t,x),\\ \label{eq:GenFormOfPointTransInEB}
\tilde A=\frac{X_x^2}{T_t}A,\quad
\tilde B=\frac{X_x}{T_t}\left(B-2\frac{U^1_x}{U^1}A\right)-\frac{X_t-X_{xx}A}{T_t},\quad
\tilde C=-\frac{U^1}{T_t}\mathrm E\frac1{U^1},\\
\label{eq:GenFormOfPointTransInEC}
\tilde D=\frac{U^1}{T_t}\left(D+\mathrm E\frac{U^0}{U^1}\right),
\end{gather}
\end{subequations}
where~$T$, $X$, $U^0$ and~$U^1$ are arbitrary smooth functions of their arguments with~$T_tX_xU^1\ne0$, and
$\mathrm E:=\p_t-A\p_{xx}-B\p_x-C$.
\end{proposition}

The normalization of the class~$\mathcal E$ means that its equivalence groupoid coincides
with the action groupoid of the pseudogroup~$G^\sim_{\mathcal E}$.

\begin{theorem}\label{thm:HeatEqSymGroup}
The point symmetry pseudogroup $G$
of the (1+1)-dimensional linear heat equation~\eqref{eq:LinHeat}
is constituted by the point transformations of the form
\begin{gather}\label{eq:HeatEqSymGroup}
\begin{split}
&\tilde t=\frac{\alpha t+\beta}{\gamma t+\delta},\quad
\tilde x=\frac{x+\lambda_1t+\lambda_0}{\gamma t+\delta},\\[.5ex]
&\tilde u=\sigma\sqrt{|\gamma t+\delta|}\,
\exp\!\left(\frac{\gamma (x+\lambda_1t+\lambda_0)^2}{4(\gamma t+\delta)}-\frac{\lambda_1}2x-\frac{\lambda_1^2}4t\right)
\big(u+h(t,x)\big),
\end{split}
\end{gather}
where $\alpha$, $\beta$, $\gamma$, $\delta$, $\lambda_1$, $\lambda_0$ and $\sigma$ are arbitrary constants with $\alpha\delta-\beta\gamma=1$ and $\sigma\ne0$,
and $h$ is an arbitrary solution of~\eqref{eq:LinHeat}. 
\end{theorem}

\begin{proof}
The linear heat equation~\eqref{eq:LinHeat}
corresponds to the value $(1,0,0,0)=:\theta^0$
of the arbitrary-ele\-ment tuple $\theta=(A,B,C,D)$ of class $\mathcal E$.
Its vertex group $\mathcal G_{\theta^0}:=\mathcal G^\sim_{\mathcal E}(\theta^0,\theta^0)$
is the set of admissible transformations of the class~$\mathcal E$ with~$\theta^0$ as both their source and target,
$\mathcal G_{\theta^0}=\{(\theta^0,\Phi,\theta^0)\mid\Phi\in G\}$.
This argument allows us to use Proposition~\ref{prop:EquivalenceGroupE} in the course of computing the pseudogroup~$G$.

We should integrate the equations~\eqref{eq:PointTransInEA},
where both the source value~$\theta$ of the arbitrary-element tuple and its target value~$\tilde\theta$
coincide with $\theta^0$,
with respect to the parameter functions $T$, $X$, $U^1$ and $U^0$.
After a simplification, the equations~\eqref{eq:GenFormOfPointTransInEB} take the form
\begin{gather}\label{eq:PointTransInETransPart}
X_x^2=T_t,\quad
\frac{U^1_x}{U^1}=-\frac{X_t}{2X_x},\quad
0=\frac{U^1}{T_t}\mathrm E\frac1{U^1},
\end{gather}
where ${\rm E}:=\p_t-\p_{xx}$.
The first equation in~\eqref{eq:PointTransInETransPart} implies that $T_t>0$,
and the first two equations in~\eqref{eq:PointTransInETransPart} can be easily integrated to
\begin{gather*}
X=\varepsilon\sqrt{T_t}\,x+X^0(t),\quad
U^1=\phi(t)\exp\left(-\frac{T_{tt}}{8T_t}x^2-\frac\varepsilon2\frac{X^0_t}{\sqrt{T_t}}x\right),
\end{gather*}
where $\varepsilon$ takes values in $\{-1,1\}$, and $\phi$ is a nonvanishing smooth function of $t$.
Substituting the expression for~$U^1$ into the third equation from~\eqref{eq:PointTransInETransPart}
and subsequently splitting the obtained equation with respect to powers of $x$,
we derive three equations, $T_{ttt}/T_t-\frac32(T_{tt}/T_t)^2=0$, $X^0_{tt}T_t-X^0_tT_{tt}=0$ and
$4T_t\phi_t+\big(T_{tt}+(X^0_t)^2\big)\phi=0$,
respectively considering them as equations for~$T$, $X^0$ and~$\phi$.
The first equation means that the Schwarzian derivative of~$T$ is zero.
Therefore, $T$ is a linear fractional function of~$t$, $T=(\alpha t+\beta)/(\gamma t+\delta)$.
Since the constant parameters~$\alpha$, $\beta$, $\gamma$ and~$\delta$ are defined up to a constant nonzero multiplier
and $T_t>0$, i.e., $\alpha\delta-\beta\gamma>0$, we can assume that $\alpha\delta-\beta\gamma=1$.
Then these parameters are still defined up to a multiplier in $\{-1,1\}$,
and hence we can choose them in such a way that $\varepsilon=\sgn(\gamma t+\delta)$.
The equation for~$X^0$ simplifies to the equation $(\gamma t+\delta)X^0_{tt}+2\gamma X^0_t=0$,
whose general solution is $X^0=(\lambda_1t+\lambda_0)/(\gamma t+\delta)$
with arbitrary constants~$\lambda_0$ and~$\lambda_1$.
The equation for~$\phi$ takes the form
${4(\gamma t+\delta)^2\phi_t-2\gamma(\gamma t+\delta)\phi+(\delta\lambda_1-\gamma\lambda_0)^2\phi=0}$
and integrates, in view of~$\phi\ne0$,
to $\phi=\sigma\sqrt{|\gamma t+\delta|}\,\exp\big(-\frac14(\delta\lambda_1-\gamma\lambda_0)X^0\big)$ with $\sigma\in\mathbb R\setminus\{0\}$.
Finally, the equation~\eqref{eq:GenFormOfPointTransInEC} takes the form $(\p_t-\p_{xx})(U^0/U^1)=0$.
Therefore, $U^0=U^1h$, where $h=h(t,x)$ is an arbitrary solution of~\eqref{eq:LinHeat}.
\end{proof}

\begin{remark}
Proposition~2 in~\cite{popo2008d} implies that
the contact symmetry pseudogroup $G_{\rm c}$ of the equation~\eqref{eq:LinHeat}
is the first prolongation of the group~$G$, $G_{\rm c}=G_{(1)}$.
\end{remark}

To avoid complicating the structure of the pseudogroup~$G$,
we should properly interpret transformations of the form~\eqref{eq:HeatEqSymGroup} and their composition.
Given a fixed transformation~$\Phi$ of the form~\eqref{eq:HeatEqSymGroup},
it is natural to assume that
its domain $\mathop{\rm dom}\Phi$ coincides with the relative complement
of the set $M_{\gamma\delta}:=\{(t,x,u)\in\mathbb R^3\mid\gamma t+\delta=0\}$ with respect to $\mathop{\rm dom}h\times\mathbb R_u$,
$\mathop{\rm dom}\Phi=(\mathop{\rm dom}h\times\mathbb R_u)\setminus M_{\gamma\delta}$.
Here $\mathop{\rm dom}F$ denotes the domain of a function~$F$.
Recall that $(\gamma,\delta)\ne(0,0)$, and note that the set~$M_{\gamma\delta}$ is
the hyperplane defined by the equation $t=-\delta/\gamma$ in $\mathbb R^3_{t,x,u}$ if $\gamma\ne0$,
and $M_{\gamma\delta}=\varnothing$ otherwise.
Instead of the standard transformation composition, we use a modified composition
for transformations of the form~\eqref{eq:HeatEqSymGroup}.
More specifically, the domain of the standard composition $\Phi_1\circ\Phi_2:=\tilde\Phi$ of transformations~$\Phi_1$ and~$\Phi_2$
is usually defined as the preimage of the domain of~$\Phi_1$ with respect to~$\Phi_2$,
$\mathop{\rm dom}\tilde\Phi=\Phi_2^{-1}(\mathop{\rm dom}\Phi_1)$.
For transformations~$\Phi_1$ and~$\Phi_2$ of the form~\eqref{eq:HeatEqSymGroup}, we have
\smash{$\mathop{\rm dom}\tilde\Phi=(\mathop{\rm dom}\tilde h\times\mathbb R_u)
\setminus (M_{\gamma_2\delta_2}\cup M_{\tilde\gamma\tilde\delta})$}
where
$\tilde\gamma=\gamma_1\alpha_2+\delta_1\gamma_2$,
$\tilde\delta=\gamma_1\beta _2+\delta_1\delta_2$,
$\mathop{\rm dom}\tilde h=\big((\pi_*\Phi_2)^{-1}\mathop{\rm dom}h^1\big)\cap\mathop{\rm dom}h^2$,
$\pi$ is the natural projection of $\mathbb R^3_{t,x,u}$ onto~$\mathbb R^2_{t,x}$,
and the parameters with indices~1 and~2 and tildes correspond~$\Phi_1$, $\Phi_2$ and~$\tilde\Phi$, respectively.
As the \emph{modified composition} $\Phi_1\circ^{\rm m}\Phi_2$ of transformations~$\Phi_1$ and~$\Phi_2$,
we take the continuous extension of $\Phi_1\circ\Phi_2$ to the set
\[\smash{\mathop{\rm dom}\nolimits^{\rm m}\tilde\Phi:=(\mathop{\rm dom}\tilde h\times\mathbb R_u)\setminus M_{\tilde\gamma\tilde\delta}},\]
i.e., \smash{$\mathop{\rm dom}(\Phi_1\circ^{\rm m}\Phi_2)=\mathop{\rm dom}^{\rm m}\tilde\Phi$}.
In other words, we set $\Phi_1\circ^{\rm m}\Phi_2$ to be the transformation of the form~\eqref{eq:HeatEqSymGroup}
with the same parameters as in $\Phi_1\circ\Phi_2$ and with the natural domain.
It is obvious that we redefine $\Phi_1\circ\Phi_2$ on the set
\smash{$(\mathop{\rm dom}\tilde h\times\mathbb R_u)\cap M_{\gamma_2\delta_2}$} if $\gamma_1\gamma_2\ne0$;
otherwise $\mathop{\rm dom}\nolimits^{\rm m}\tilde\Phi=\mathop{\rm dom}\tilde\Phi$
and the extension is trivial.
A disadvantage of the above interpretation is that
it is then common for elements of~$G$ to have different signs of their Jacobians
on different connected components of their domains,
but the benefits we receive due to it are more essential.

Now we can analyze the structure of~$G$.
The point transformations of the form
\[
\mathscr Z(f)\colon\quad \tilde t=t,\quad \tilde x=x,\quad \tilde u=u+f(t,x),
\]
where the parameter function $f=f(t,x)$ is an arbitrary solution of the equation~\eqref{eq:LinHeat},
are associated with the linear superposition of solutions of this equation
and, thus, can be considered as trivial.
They constitute the normal pseudosubgroup~$G^{\rm lin}$ of the pseudogroup~$G$.
The pseudogroup~$G$ splits over~$G^{\rm lin}$, $G=G^{\rm ess}\ltimes G^{\rm lin}$,
where the subgroup~$G^{\rm ess}$ of~$G$ consists of the transformations of the form~\eqref{eq:HeatEqSymGroup} with $f=0$
and thus is a six-dimensional Lie group.
We call the subgroup~$G^{\rm ess}$ the \emph{essential point symmetry group} of the equation~\eqref{eq:LinHeat}.
This subgroup itself splits over $R$, $G^{\rm ess}=F\ltimes R$.
Here $R$ and~$F$ are the normal subgroup and the subgroup of~$G^{\rm ess}$
that are singled out by the constraints $\alpha=\delta=1$, $\beta=\gamma=0$ and $\lambda_1=\lambda_0=0$, $\sigma=1$,
respectively.
They are isomorphic to the groups ${\rm H}(1,\mathbb R)\times\mathbb Z_2$ and ${\rm SL}(2,\mathbb R)$,
and their Lie algebras coincide with~$\mathfrak r\simeq{\rm h}(1,\mathbb R)$ and~$\mathfrak f\simeq{\rm sl}(2,\mathbb R)$.
Here ${\rm H}(1,\mathbb R)$ denotes the rank-one real Heisenberg group.
The normal subgroups~$R_{\rm c}$ and~$R_{\rm d}$ of~$R$ that are isomorphic to~${\rm H}(1,\mathbb R)$ and~$\mathbb Z_2$
are constituted by the elements of~$R$ with parameter values satisfying the constraints
$\sigma>0$ and $\lambda_0=\lambda_1=0$, $\sigma\in\{-1,1\}$, respectively.
The isomorphisms of~$F$ to ${\rm SL}(2,\mathbb R)$ and of~$R_{\rm c}$ to~${\rm H}(1,\mathbb R)$ are established by the correspondences
\[
\varrho_1=(\alpha,\beta,\gamma,\delta)_{\alpha\delta-\beta\gamma=1}\mapsto
\begin{pmatrix}
\alpha & \beta \\
\gamma  & \delta
\end{pmatrix},
\qquad
(\lambda_1,\lambda_0,\sigma)_{,\;\sigma>0}\mapsto
\begin{pmatrix}
1 & -\tfrac12\lambda_1 & \ln\sigma\\
0 & 1                  & \lambda_0\\
0 & 0                  & 1
\end{pmatrix}.
\]
The isomorphism $\varrho_1$ is in fact the standard two-dimensional representation of ${\rm SL}(2,\mathbb R)$.  
Thus, $F$ and~$R_{\rm c}$ are connected subgroups of~$G^{\rm ess}$, but~$R_{\rm d}$ is not.
The natural conjugacy action of the subgroup~$F$ on the normal subgroup~$R$ is given,
in the parameterization~\eqref{eq:HeatEqSymGroup} of~$G$, by
$(\tilde\lambda_1,\tilde\lambda_0,\tilde\sigma)=(\lambda_1,\lambda_0,\sigma)A$
with the matrix $A=\varrho_1(\alpha,\beta,\gamma,\delta)\oplus(1)$.
Summing up, we have that
\[G^{\rm ess}\simeq \big({\rm SL}(2,\mathbb R)\ltimes_{\varphi}{\rm H}(1,\mathbb R)\big)\times\mathbb Z_2,\]
where the antihomomorphism
$\varphi\colon{\rm SL}(2,\mathbb R)\to{\rm Aut}({\rm H}(1,\mathbb R))$ is defined, in the chosen local coordinates, by
$
\varphi(\alpha,\beta,\gamma,\delta)=(\lambda_1,\lambda_0,\sigma)
\mapsto(\alpha\lambda_1+\gamma\lambda_0,\beta\lambda_1+\delta\lambda_0,\sigma)=(\lambda_1,\lambda_0,\sigma)A.
$

The transformations from the one-parameter subgroups of~$G^{\rm ess}$ that are generated by the basis elements of~$\mathfrak g^{\rm ess}$
given in~\eqref{eq:LinHeatEssAlg} are of the following form:
\[\arraycolsep=2ex
\begin{array}{llll}
\mathscr P^t(\epsilon)\colon    & \tilde t=t+\epsilon,            & \tilde x=x,                   & \tilde u=u,\\[.6ex]
\mathscr D  (\epsilon)\colon    & \tilde t={\rm e}^{2\epsilon}t,  & \tilde x={\rm e}^\epsilon x,  & \tilde u={\rm e}^{-\frac12\epsilon}u,\\
\mathscr K(\epsilon)  \colon    & \tilde t=\dfrac t{1-\epsilon t},
& \tilde x=\dfrac x{1-\epsilon t}, &\tilde u=\sqrt{|1-\epsilon t|}{\rm e}^{\frac{\epsilon x^2}{4(\epsilon t-1)}}u,\\[1ex]
\mathscr G^x(\epsilon)\colon    & \tilde t=t,              & \tilde x=x+\epsilon t,    & \tilde u={\rm e}^{-\frac14(\epsilon^2t+2\epsilon x)}u,\\[1ex]
\mathscr P^x(\epsilon)\colon    & \tilde t=t,              & \tilde x=x+\epsilon,      & \tilde u=u,\\[1ex]
\mathscr I  (\epsilon)\colon    & \tilde t=t,              & \tilde x=x,               & \tilde u={\rm e}^\epsilon u,
\end{array}
\]
where $\epsilon$ is an arbitrary constant.
At the same time, using this basis of~$\mathfrak g^{\rm ess}$ in the course of studying the structure of the group~$G^{\rm ess}$
hides some of its important properties and complicates its study.

Although the pushforward of the pseudogroup~$G$ by the natural projection of~$\mathbb R^3_{t,x,u}$ onto~$\mathbb R_t$
coincides with the group of linear fractional transformations of~$t$ and is thus isomorphic to the group ${\rm PSL}(2,\mathbb R)$,
the subgroup~$F$ of~$G$ is isomorphic to the group ${\rm SL}(2,\mathbb R)$,
and its Iwasawa decomposition is given by the one-parameter subgroups of~$G$
respectively generated by the vector fields~$\mathcal Q^+:=\mathcal P^t+\mathcal K$, $\mathcal D$ and $\mathcal P^t$.
The first subgroup, which is associated with~$\mathcal Q^+$, consists of the point transformations
\[
\mathscr Q^+(\epsilon)\colon\ \
\tilde t=\frac{\sin\epsilon+t\cos\epsilon}{\cos\epsilon-t\sin\epsilon},
\quad
\tilde x=\frac x{\cos\epsilon-t\sin\epsilon},
\quad
\tilde u=\sqrt{|\cos\epsilon-t\sin\epsilon|}\,{\rm e}^{\frac{-x^2\sin\epsilon}{4(\cos\epsilon-t\sin\epsilon)}
}u,
\]
where $\epsilon$ is an arbitrary constant parameter, which is defined by the corresponding transformation
up to a summand $2\pi k$, $k\in\mathbb Z$.
The Jacobian of~$\mathscr Q^+(\epsilon)$ is positive and negative for all values of $(t,x,u)$
if $\epsilon=0$ and $\epsilon=\pi$, respectively.
For $\epsilon\in(0,\pi)\cup(\pi,2\pi)$, the transformation~$\mathscr Q^+(\epsilon)$ is not defined if $t=\cot\epsilon$,
and for the other values of $(t,x,u)$ the sign of its Jacobian coincides with $\sgn(\cos\epsilon-t\sin\epsilon)$.

The equation~\eqref{eq:LinHeat} is invariant with respect to the involution~$\mathscr J$ alternating the sign of~$x$
and the transformation~$\mathscr K'$ inverting~$t$,
\begin{gather}\label{eq:LinHeatTransJK'}
\smash{\mathscr J:=(t,x,u)\mapsto(t,-x,u),\qquad
\mathscr K'\colon\ \ \tilde t=-\dfrac1t,\ \ \tilde x=\dfrac xt,\ \
\tilde u=\sqrt{|t|}\,{\rm e}^{\frac{x^2}{4t}}u.}
\end{gather}
Note that $(\mathscr K')^2=\mathscr J$.
In the context of the one-parameter subgroups of~$G^{\rm ess}$ (resp.\ of~$G$)
that are generated by the basis elements of~$\mathfrak g^{\rm ess}$ listed in~\eqref{eq:LinHeatEssAlg},
both the transformations~$\mathscr J$ and~$\mathscr K'$ look as discrete point symmetry transformations of~\eqref{eq:LinHeat},
but in fact, this is not the case under the above interpretation of the group multiplication in~$G$.
Even though the Jacobian of the involution~$\mathscr J$ is equal to $-1$ for all values of $(t,x,u)$,
this involution belongs to the one-parameter subgroup~$\{\mathscr Q^+(\epsilon)\}$ of~$G$,
$\mathscr J=\mathscr Q^+(\pi)$, and thus it belongs to the identity component of the pseudogroup~$G$.
A similar claim is true for the transformation~$\mathscr K'=\mathscr Q^+(-\frac12\pi)$, the sign of whose Jacobian coincides with $\sgn t$.

\begin{corollary}\label{cor:LinHeatDiscretePointSyms}
A complete list of discrete point symmetry transformations of the linear (1+1)-dimensional heat equation~\eqref{eq:LinHeat}
that are independent up to composing with each other and with continuous point symmetry transformations of this equation
is exhausted by the single involution~$\mathscr I'$ alternating the sign of~$u$,
\[\mathscr I':=(t,x,u)\mapsto(t,x,-u).\]
Thus, the factor group of the point symmetry pseudogroup~$G$
with respect to its identity component~$G_{\rm id}$ is isomorphic to~$\mathbb Z_2$.
\end{corollary}

\begin{proof}
It is obvious that the entire pseudosubgroup~$G^{\rm lin}$ is contained in the connected component of the identity transformation in~$G$.
The same claim holds for the subgroups~$F$ and~$R_{\rm c}$
in view of their isomorphisms to the groups ${\rm SL}(2,\mathbb R)$ and~${\rm H}(1,\mathbb R)$, respectively.
Therefore, without loss of generality, a complete list of independent discrete point symmetry transformations of~\eqref{eq:LinHeat}
can be assumed to consist of elements of the subgroup~$R_{\rm d}$.
Thus, the only discrete point symmetry transformation of~\eqref{eq:LinHeat}
that is independent in the above sense is the transformation~$\mathscr I'$,
and the identity component~$G_{\rm id}$ of~$G$ is constituted by transformations of the form~\eqref{eq:HeatEqSymGroup}
with $\sigma>0$.
\end{proof}

\begin{corollary}\label{cor:LinHeatZenterOfGess}
The center~${\rm Z}(G^{\rm ess})$ of the group~$G^{\rm ess}$
coincides with $\{\mathscr I(\epsilon)\}\sqcup\{\mathscr I'\circ\mathscr I(\epsilon)\}$.
\end{corollary}

\begin{proof}
Given a group~$G$, by~${\rm Z}(G)$ we denote its center.
The decomposition $G^{\rm ess}=(F\ltimes R_{\rm c})\times R_{\rm d}$
and the obvious inclusion $R_{\rm d}\subseteq{\rm Z}(G^{\rm ess})$
jointly imply that ${\rm Z}(G^{\rm ess})=\big({\rm Z}(G^{\rm ess})\cap(F\ltimes R_{\rm c})\big)\times R_{\rm d}$.
Since $R_{\rm c}\simeq{\rm H}(1,\mathbb R)$, we have ${\rm Z}(R_{\rm c})=\{\mathscr I(\varepsilon)\}$.
It is easy to check that $\{\mathscr I(\varepsilon)\}\subseteq{\rm Z}(G^{\rm ess})$,
and thus ${\rm Z}(G^{\rm ess})\cap R_{\rm c}=\{\mathscr I(\varepsilon)\}$.
For any element from~${\rm Z}(G^{\rm ess})$, its $F$-component belongs to~${\rm Z}(F)$.
Since $F\simeq{\rm SL}(2,\mathbb R)$,
we have that ${\rm Z}(F)=\{{\rm id},\mathscr J\}$, where ${\rm id}$ is the identity element of~$G^{\rm ess}$.
However, $\mathscr J\circ\Phi\notin {\rm Z}(G^{\rm ess})$ for any $\Phi\in R_{\rm c}$
since, e.g., $\mathscr J\circ\Phi\circ\mathscr P^x(\epsilon)\ne\mathscr P^x(\epsilon)\circ\mathscr J\circ\Phi$
for $\epsilon\ne0$.
Therefore, ${\rm Z}(G^{\rm ess})={\rm Z}(R_{\rm c})\times R_{\rm d}$.
\end{proof}

In the notation of Theorem~\ref{thm:HeatEqSymGroup},
the most general form of the transformed counterpart of a given solution $u=f(t,x)$
under action of~$G$ (resp.\ $G_{\rm id}$) is
\begin{gather*}
\tilde u=\frac{\sigma}{\sqrt{|\gamma t-\alpha|}}\,
\exp\!\left(\frac{\gamma x^2}{4(\alpha-\gamma t)}
-\frac{\lambda_1x}{2(\alpha-\gamma t)}
+\frac{\lambda_1^2}4\frac{\delta t-\beta}{\alpha-\gamma t}
+\frac{\lambda_0\lambda_1}2
\right)
\\
\hphantom{\tilde u=}
\times f\left(\frac{\delta t-\beta}{\alpha-\gamma t},
\frac x{\alpha-\gamma t}-\lambda_1\frac{\delta t-\beta}{\alpha-\gamma t}-\lambda_0\right)
+h(t,x),
\end{gather*}
where in addition $\sigma>0$ for~$G_{\rm id}$, cf., e.g., \cite[p.~120]{olve1993A}.

Note that in view of Theorem~\ref{thm:HeatEqSymGroup},
formally pulling back the function~$u:=(t,x)\mapsto H(-t)$, where $H(t)$ denotes the Heaviside step function,
by the point symmetry transformation
$\Phi:=\mathscr J(\ln\sqrt{4\pi})\circ\mathscr K'$,
\[
\Phi\colon\quad
\tilde t=-\frac1t,\quad
\tilde x=\frac xt,\quad
\tilde u=\sqrt{4\pi|t|}\exp\left(\frac{x^2}{4t}\right)u,
\]
we obtain the well-known fundamental solution~$F$ of~\eqref{eq:LinHeat},
\smash{$F(t,x)=\dfrac{H(t)}{\sqrt{4\pi|t|}}\exp\left(-\dfrac{x^2}{4t}\right)$}.

\section{Classifications of subalgebras}\label{sec:LinHeatClassIneqSubalgebras}

In spite of the unusualness of Corollary~\ref{cor:LinHeatDiscretePointSyms} and the claims in the paragraph before it,
they are well consistent
with the structure of the abstract Lie group that is isomorphic to~$G^{\rm ess}$ and
with the inner automorphism group ${\rm Inn}(\mathfrak g^{\rm ess})$ of~$\mathfrak g^{\rm ess}$.
More specifically, the algebra~$\mathfrak g^{\rm ess}$ is the Lie algebra of the group~$G^{\rm ess}$,
and the adjoint action of~$G^{\rm ess}$ on~$\mathfrak g^{\rm ess}$ coincides with ${\rm Inn}(\mathfrak g^{\rm ess})$.
In particular, under the suggested interpretation
we have ${\rm Ad}(\exp(\epsilon\mathcal Q^+))=\mathscr Q^+(\epsilon)_*$,
\begin{gather*}
\begin{array}{l}
\mathscr Q^+(\epsilon)_*\mathcal Q^-=\hphantom{-{}}\cos(2\epsilon)\mathcal Q^-+\sin(2\epsilon)\mathcal D,\\[.5ex]
\mathscr Q^+(\epsilon)_*\mathcal D\hphantom{^-}=-\sin(2\epsilon)\mathcal Q^-+\cos(2\epsilon)\mathcal D,
\end{array}\quad
\begin{array}{l}
\mathscr Q^+(\epsilon)_*\mathcal P^x=\hphantom{-{}}\cos(\epsilon)\mathcal P^x+\sin(\epsilon)\mathcal G^x,\\[.5ex]
\mathscr Q^+(\epsilon)_*\mathcal G^x=-\sin(\epsilon)\mathcal P^x+\cos(\epsilon)\mathcal G^x,
\end{array}
\end{gather*}
and the inner automorphisms associated with $\mathscr J=\mathscr Q^+(\pi)$ and $\mathscr K'=\mathscr Q^+(-\frac12\pi)$
allow one to map $\mathcal P^t-\mathcal G^x$ and $D-\mu\mathcal I$ to
$\mathcal P^t+\mathcal G^x$ and $D+\mu\mathcal I$, respectively.
Retaining these facts, we refine the classification of subalgebras of~$\mathfrak g^{\rm ess}$
or, equivalently, the special Galilei algebra
\[{\rm sl}(2,\mathbb R)\lsemioplus_{\rho_1\oplus \rho_0}{\rm h}(1,\mathbb R),\]
cf.\ \cite[Example~3.13]{olve1993A}, \cite[p.~531--535]{wint1990a} and~\cite{chou2001c}.%
\footnote{%
The closest to a correct optimal list of subalgebras of~$\mathfrak g^{\rm ess}$
is the list presented in~\cite{chou2001c}.
The only its inconvenience is
the missed constraint $\mu\in\mathbb R_{\geqslant0}$ from Lemma~\ref{lem:LinHeat1DSubalgs}.
}
The nonidentity pushforwards of the basis elements of $\mathfrak g^{\rm ess}$
by the elementary transformations from $G^{\rm ess}$ discussed in Section~\ref{sec:LinHeatPointSymGroup}
are, in addition to the above ones by $\mathscr Q^+(\epsilon)$, the following:
\begin{gather*}\arraycolsep=0ex
\begin{array}{l}
\mathscr P^t(\epsilon)_*\mathcal D  =\mathcal D-2\epsilon\mathcal P^t,\\[.75ex]
\mathscr P^t(\epsilon)_*\mathcal K  =\mathcal K-\epsilon\mathcal D+\epsilon^2\mathcal P^t,\\[.75ex]
\mathscr P^t(\epsilon)_*\mathcal G^x=\mathcal G^x-\epsilon\mathcal P^x,
\end{array}
\qquad
\begin{array}{l}
\mathscr K(\epsilon)_*\mathcal D  =\mathcal D+2\epsilon\mathcal K,\\[.75ex]
\mathscr K(\epsilon)_*\mathcal P^t=\mathcal P^t+\epsilon\mathcal D+\epsilon^2\mathcal K,\\[.75ex]
\mathscr K(\epsilon)_*\mathcal P^x=\mathcal P^x+\epsilon\mathcal G^x,
\end{array}
\\[2ex]\arraycolsep=0ex
\begin{array}{l}
\mathscr D(\epsilon)_*\mathcal P^t=e^{2\epsilon}\mathcal P^t,\\[1ex]	
\mathscr D(\epsilon)_*\mathcal K  =e^{-2\epsilon}\mathcal K,
\end{array}\qquad
\begin{array}{l}
\mathscr D(\epsilon)_*\mathcal G^x=e^{-\epsilon}\mathcal G^x,\\[1ex]
\mathscr D(\epsilon)_*\mathcal P^x=e^{\epsilon}\mathcal P^x,
\end{array}
\\[2ex]\arraycolsep=0ex	
\begin{array}{l}
\mathscr G^x(\epsilon)_*\mathcal P^t=\mathcal P^t+\epsilon\mathcal P^x-\tfrac14\epsilon^2\mathcal I,\\[1ex]
\mathscr G^x(\epsilon)_*\mathcal D  =\mathcal D+\epsilon\mathcal G^x,\\[1ex]
\mathscr G^x(\epsilon)_*\mathcal P^x=\mathcal P^x-\tfrac12\epsilon\mathcal I,
\end{array}
\qquad
\begin{array}{l}
\mathscr P^x(\epsilon)_*\mathcal D  =\mathcal D-\epsilon\mathcal P^x,\\[1ex]
\mathscr P^x(\epsilon)_*\mathcal K  =\mathcal K-\epsilon\mathcal G^x-\tfrac14\epsilon^2\mathcal I,\\[1ex]
\mathscr P^x(\epsilon)_*\mathcal G^x=\mathcal G^x+\tfrac12\epsilon\mathcal I,
\end{array}
\\[2ex]
\mathscr J_*(\mathcal G^x,\mathcal P^x)=(-\mathcal G^x,-\mathcal P^x),
\qquad
\mathscr K'_*(\mathcal P^t,\mathcal D,\mathcal K,\mathcal G^x,\mathcal P^x)
=(\mathcal K,-\mathcal D,\mathcal P^t,\mathcal P^x,-\mathcal G^x).
\end{gather*}

\begin{lemma}\label{lem:LinHeat1DSubalgs}
A complete list of inequivalent proper subalgebras of the algebra~$\mathfrak g^{\rm ess}$
is exhausted by the following subalgebras, where $\delta\in\{-1,0,1\}$, $\mu\in\mathbb R_{\geqslant0}$ and $\nu\in\mathbb R$:
\begin{gather*}
1{\rm D}\colon\
\mathfrak s_{1.1}       =\langle\mathcal P^t+\mathcal G^x\rangle,\ \
\mathfrak s_{1.2}^\delta=\langle\mathcal P^t+\delta\mathcal I\rangle,\ \
\mathfrak s_{1.3}^\mu   =\langle\mathcal D+\mu\mathcal I\rangle,\ \
\mathfrak s_{1.4}^\nu   =\langle\mathcal P^t+\mathcal K+\nu\mathcal I\rangle,\\
\hphantom{1{\rm D}\colon\ }
\mathfrak s_{1.5}       =\langle\mathcal P^x\rangle,\ \
\mathfrak s_{1.6}       =\langle\mathcal I\rangle,
\\[1ex]
2{\rm D}\colon\
\mathfrak s_{2.1}^\nu   =\langle\mathcal P^t,\mathcal D+\nu\mathcal I\rangle,\ \
\mathfrak s_{2.2}       =\langle\mathcal P^t+\mathcal G^x,\mathcal I \rangle,\ \
\mathfrak s_{2.3}^\delta=\langle\mathcal P^t+\delta\mathcal I,\mathcal P^x \rangle,\ \
\mathfrak s_{2.4}       =\langle\mathcal P^t,\mathcal I \rangle,\\
\hphantom{2{\rm D}\colon\ }
\mathfrak s_{2.5}^\nu   =\langle\mathcal D+\nu\mathcal I,\mathcal P^x\rangle,\ \
\mathfrak s_{2.6}       =\langle\mathcal D,\mathcal I\rangle,\ \
\mathfrak s_{2.7}       =\langle\mathcal P^t+\mathcal K,\mathcal I\rangle,\ \
\mathfrak s_{2.8}       =\langle\mathcal P^x,\mathcal I\rangle,
\\[1ex]
3{\rm D}\colon\
\mathfrak s_{3.1}    =\langle\mathcal P^t,\mathcal D,\mathcal K\rangle,\ \
\mathfrak s_{3.2}^\nu=\langle\mathcal P^t,\mathcal D+\nu\mathcal I,\mathcal P^x\rangle,\ \
\mathfrak s_{3.3}    =\langle\mathcal P^t,\mathcal D,\mathcal I\rangle,\\
\hphantom{3{\rm D}\colon\ }
\mathfrak s_{3.4}    =\langle\mathcal P^t+\mathcal G^x,\mathcal P^x,\mathcal I\rangle,\ \
\mathfrak s_{3.5}    =\langle\mathcal P^t,\mathcal P^x,\mathcal I\rangle,\ \
\mathfrak s_{3.6}    =\langle\mathcal D,\mathcal P^x,\mathcal I\rangle,\ \
\mathfrak s_{3.7}    =\langle\mathcal G^x,\mathcal P^x,\mathcal I\rangle,
\\[1ex]
4{\rm D}\colon\
\mathfrak s_{4.1}    =\langle\mathcal P^t,\mathcal D,\mathcal K,\mathcal I\rangle,\ \
\mathfrak s_{4.2}    =\langle\mathcal P^t,\mathcal D,\mathcal P^x,\mathcal I\rangle,\ \
\mathfrak s_{4.3}    =\langle\mathcal P^t,\mathcal G^x,\mathcal P^x,\mathcal I\rangle,\\
\hphantom{4{\rm D}\colon\ }
\mathfrak s_{4.4}    =\langle\mathcal D,\mathcal G^x,\mathcal P^x,\mathcal I\rangle,\ \
\mathfrak s_{4.5}    =\langle\mathcal P^t+\mathcal K,\mathcal G^x,\mathcal P^x,\mathcal I\rangle,
\\[1ex]
5{\rm D}\colon\
\mathfrak s_{5.1}    =\langle\mathcal P^t,\mathcal D,\mathcal G^x,\mathcal P^x,\mathcal I\rangle.
\end{gather*}
\end{lemma}

\begin{proof}
In the course of classifying $G^{\rm ess}$-inequivalent subalgebras of~$\mathfrak g^{\rm ess}$
we use its Levi decomposition, $\mathfrak g^{\rm ess}=\mathfrak f\lsemioplus\mathfrak r$,
and the fact that $\mathfrak f\simeq{\rm sl}(2,\mathbb R)$.
The technique for classification subalgebras of Lie algebras
whose Levi factors are isomorphic to~${\rm sl}(2,\mathbb R)$ is developed in~\cite[Section 4]{kova2022a} and~\cite{wint1990a}
on the base of~\cite{pate1975a}.
In particular, this technique involves the fact
that the subalgebras~$\mathfrak s_1$ and~$\mathfrak s_2$ of~$\mathfrak g^{\rm ess}$ are definitely~$G^{\rm ess}$-inequivalent
if their projections on the Levi factor~$\mathfrak f$,
which are subalgebras~$\pi_{\mathfrak f}\mathfrak s_1$ and~$\pi_{\mathfrak f}\mathfrak s_2$,
are $F$-inequivalent.
Here~$\pi_{\mathfrak f}$ denotes the natural projection of~$\mathfrak g^{\rm ess}$ on~$\mathfrak f$
according to the decomposition $\mathfrak g^{\rm ess}=\mathfrak f\dotplus\mathfrak r$ of~$\mathfrak g^{\rm ess}$
as a vector space into the complementary subspaces~$\mathfrak f$ and~$\mathfrak r$.
Another simplification is that an optimal list of subalgebras for ${\rm sl}(2,\mathbb R)$ is well known,
and for the realization~$\mathfrak f$ of~${\rm sl}(2,\mathbb R)$ it is constituted by the subalgebras
$\{0\}$, $\langle\mathcal P^t\rangle$,
$\langle\mathcal D\rangle$,
$\langle\mathcal P^t+\mathcal K\rangle$,
$\langle\mathcal P^t, \mathcal D\rangle$
and $\mathfrak f$ itself.
To distinguish $G^{\rm ess}$-inequivalent subalgebras, we can also use the obvious $G^{\rm ess}$-invariants
$n:=\dim\mathfrak s$,
$\hat n:=\dim\pi_{\mathfrak f}\mathfrak s$,
$\check n:=\dim\mathfrak s\cap\mathfrak r=n-\hat n$
and $\check n':=\dim\mathfrak s\cap\mathfrak r'$.
Note that $\dim\mathfrak f=3$ and $\dim\mathfrak r=3$, and hence
$0\leqslant \hat n\leqslant\min(3,n)$ and $0\leqslant \check n\leqslant\min(3,n)$.

This is why, modulo the $G^{\rm ess}$-equivalence,
we split the classification of subalgebras~$\mathfrak s$ of the algebra~$\mathfrak g^{\rm ess}$ into cases
depending on equivalence classes of~$\pi_{\mathfrak f}\mathfrak s$
as well as on the value of the $G^{\rm ess}$-invariants $n$, $\hat n$, $\check n$ and $\check n'$.
For each of the cases, we fix a subalgebra in the above list of inequivalent subalgebras of~$\mathfrak f$
as a canonical representative for~$\pi_{\mathfrak f}\mathfrak s$,
$\pi_{\mathfrak f}\mathfrak s=\langle \hat Q_1,\dots,\hat Q_{\hat n}\rangle$ with
\[
(\hat Q_1,\dots,\hat Q_{\hat n})\in\big\{(),\,(\mathcal P^t),\,(\mathcal D),\,(\mathcal P^t+\mathcal K),\,
(\mathcal P^t,\mathcal D),\,(\mathcal P^t,\mathcal D,\mathcal K)\big\}.
\]
Then we consider a basis of~$\mathfrak s$ consisting of the vector fields of the form
$Q_i=\hat Q_i+a_i\mathcal P^x+b_i\mathcal G^x+c_i\mathcal I$, $i=1,\dots,n$,
where $\hat Q_i=0$, $i>\hat n$.
We should further simplify the form of~$Q_i$ or, equivalently,
set the constant parameters~$a_i$, $b_i$ and~$c_i$ to as simpler values as possible
via linearly recombining~$Q_i$ and pushing $\mathfrak g^{\rm ess}$ forward by elements of~$G^{\rm ess}$.
For $n\geqslant2$, we should take into account
the constraints for the parameters~$a_i$, $b_i$ and~$c_i$
implied by the closedness of the subalgebra~$\mathfrak s$ with respect to the Lie bracket of vector fields,
$[Q_i,Q_j]\in\mathfrak s$, $i,j\in\{1,\dots,n\}$.

\medskip\par\noindent	
$\boldsymbol{\check n=3.}$
Thus, $\mathfrak s\supseteq\mathfrak r$,
and we can choose $(Q_{n-2},Q_{n-1},Q_n)=(\mathcal P^x,\mathcal G^x,\mathcal I)$.
Linearly combining $Q_i$, $i=1,\dots,\hat n$, with $Q_{n-2}$, $Q_{n-1}$ and~$Q_n$,
we set $a_i=b_i=c_i=0$, $i=1,\dots,\hat n$, which then means that
$\pi_{\mathfrak f}\mathfrak s=\mathfrak s\cap\mathfrak f$.
The span of~$\mathfrak r$ and any subalgebra of~$\mathfrak f$
is necessarily closed with respect to the Lie bracket.
Therefore, in this case, we obtain the proper subalgebras~$\mathfrak s_{3.7}$,
$\mathfrak s_{4.3}$, $\mathfrak s_{4.4}$, $\mathfrak s_{4.5}$ and~$\mathfrak s_{5.1}$.

\medskip\par
Below we assume {\mathversion{bold}$\check n<3$} and also use the following observation.
If $\check n=2$, then $(a_{n-1},b_{n-1})$ and $(a_n,b_n)$ are linearly dependent
since otherwise $[Q_{n-1},Q_n]=(a_nb_{n-1}-a_{n-1}b_n)\mathcal I\in\mathfrak s$
and thus $\check n=3$, which contradicts the supposition $\check n=2$.
Linearly recombining~$Q_{n-1}$ and~$Q_n$, we can set
$(a_{n-1},b_{n-1})\ne(0,0)$ and $Q_n=\mathcal I$ in this case.
In general, if $\check n'=1$, i.e., $\mathcal I\in\mathfrak s$, then we choose $Q_n=\mathcal I$ and set
the coefficients~$c_i$, $i=1,\dots,n-1$ to zero
by linearly combining the other basis elements with~$Q_n$.

\medskip\par\noindent	
$\boldsymbol{\pi_{\mathfrak f}\mathfrak s=\{0\}.}$
Hence $\hat n=0$ and $\mathfrak s\subseteq\mathfrak r$.
If $(a_1,b_1)\ne(0,0)$, then we can set $b_1=0$,
pushing $\mathfrak g^{\rm ess}$ forward by~$\mathscr Q^+(\epsilon)$
with an appropriate~$\epsilon$.
Since the new value of the parameter~$a_1$ is necessarily nonzero,
we divide~$Q_1$ by~$a_1$ to set $a_1=1$.
Then the pushforward~$\mathscr G^x(2c_1)_*$ allows us to set $c_1=0$.
Therefore, any subalgebra of this case is $G^{\rm ess}$-equivalent
to one of the subalgebras~$\mathfrak s_{1.5}$, $\mathfrak s_{1.6}$ and~$\mathfrak s_{2.8}$.

\medskip\par\noindent	
$\boldsymbol{\pi_{\mathfrak f}\mathfrak s=\langle\mathcal P^t\rangle.}$
Acting by~$\mathscr G^x(-a_1)_*$, we set $a_1=0$.

In the case $b_1\ne0$, we first set $b_1>0$ using, if necessary, $\mathscr J_*$
and then scale $b_1$ to~1 using $\mathscr D(\epsilon)_*$ with an appropriate~$\epsilon$
and scaling the entire~$Q_1$.
The next reduction is to set $c_1=0$ by $\mathscr P^x(-2c_1)_*$.
If $n>1$, then $[Q_1,Q_2]=b_2\mathcal P^x+\tfrac12a_2\mathcal I$.
Hence we derive the subalgebras~$\mathfrak s_{1.1}$, $\mathfrak s_{2.2}$ and $\mathfrak s_{3.4}$
if $\check n=0$, $\check n=1$ and $\check n=2$, respectively.

Let now $b_1=0$.
If simultaneously $c_1\ne0$, then we can set $|c_1|=1$
successively acting by $\mathscr D(\tfrac12\ln|c_1|)_*$ on~$\mathfrak g^{\rm ess}$
and dividing~$Q_1$ by~$|c_1|$.
If $n>1$, then $[Q_1,Q_2]=b_2\mathcal P^x$.
The corresponding subalgebras for
$\check n=0$, $(\check n,\check n')=(1,0)$, $(\check n,\check n')=(1,1)$ and $\check n=2$
are
$\mathfrak s_{1.2}^\delta$, $\mathfrak s_{2.3}^\delta$, $\mathfrak s_{2.4}$ and~$\mathfrak s_{3.5}$.
In particular, the chain of simplifications in the case $(\check n,\check n')=(1,0)$ is the following.
Since $b_2=0$ and thus $a_2\ne0$ in this case, we divide~$Q_2$ by~$a_2$
and push the new~$Q_2$ forward by~$\mathscr G^x(2c_2)_*$ to set $a_2=1$ and $c_2=0$, respectively.
Then we re-establish the gauges $a_1=0$ and $c_1\in\{-1,0,1\}$
by combining~$Q_1$ with~$Q_2$ and, if the new $c_1$ is nonzero, repeating
the action by $\mathscr D(\tfrac12\ln|c_1|)_*$ and the division of~$Q_1$ by~$|c_1|$.

\medskip\par\noindent	
$\boldsymbol{\pi_{\mathfrak f}\mathfrak s=\langle\mathcal D\rangle.}$
Acting by~$\mathscr P^x(a_1)_*$ and~$\mathscr G^x(-b_1)_*$, we set $a_1=b_1=0$.
For $n=1$, the cases $c_1>0$ and~$c_1<0$ are related
via~$\mathscr K'_*$ and multiplying~$Q_1$ by~$-1$,
which leads to the subalgebra~$\mathfrak s_{1.3}^\mu$.
For $n>1$, we have $[Q_1,Q_2]=-a_2\mathcal P^x+b_2\mathcal G^x\in\mathfrak s$,
and thus $a_2b_2=0$ since $\check n<3$.
Moreover, if $(a_2,b_2)\ne(0,0)$, then the coefficient~$c_2$ can be assumed to be equal to zero
since there is no summand with~$\mathcal I$ in $[Q_1,Q_2]$.
The case $b_2\ne0$ is mapped to the case $a_2\ne0$ by~$\mathscr K'_*$.
This gives the subalgebras~$\mathfrak s_{2.5}^\nu$, $\mathfrak s_{2.6}$ and~$\mathfrak s_{3.6}$.

\medskip\par\noindent	
$\boldsymbol{\pi_{\mathfrak f}\mathfrak s=\langle\mathcal P^t+\mathcal K\rangle.}$
First we set $a_1=b_1=0$, acting by~$\mathscr P^x(b_1)_*$ and~$\mathscr G^x(-a_1)_*$.
For $n=1$, there is no possible further simplification of~$\mathfrak s$.
For $n>1$, we obtain $[Q_1,Q_2]=b_2\mathcal P^x-a_2\mathcal G^x$,
which implies $a_2=b_2=0$ in view of $\check n<3$.
As a result, we derive the subalgebras~$\mathfrak s_{1.4}^\nu$ and~$\mathfrak s_{2.7}$.

\medskip\par\noindent	
$\boldsymbol{\pi_{\mathfrak f}\mathfrak s=\langle\mathcal P^t, \mathcal D\rangle.}$
We reduce~$Q_2$ to the form $Q_2=\mathcal D+\nu\mathcal I$,
following the consideration in the case $\pi_{\mathfrak f}\mathfrak s=\langle\mathcal D\rangle$.
Consider the commutators
$[Q_1,Q_2]=2\mathcal P^t+a_1\mathcal P^x-b_1\mathcal G^x$
and, if $n>2$,
$[Q_1,Q_3]=b_3\mathcal P^x+\frac12(a_3b_1-a_1b_3)\mathcal I$ and
$[Q_2,Q_3]=-a_3\mathcal P^x+b_3\mathcal G^x$.
For $n=2$, it follows from the condition $[Q_1,Q_2]\in\mathfrak s$ that $a_1=b_1=c_1=0$,
which gives the algebra~$\mathfrak s_{2.1}^\nu$.
Let $n>2$.
The conditions $[Q_1,Q_3]\in\mathfrak s$ and $\check n<3$ jointly implies that $b_3=0$.
Then, the condition $[Q_1,Q_2]\in\mathfrak s$ allows us to conclude that $b_1=0$,
that the coefficient~$c_1$ can be assumed to be equal to zero
since there is no summand with~$\mathcal I$ in $[Q_1,Q_2]$
and that $a_1=0$ if $a_3=0$.
For $a_3\ne0$, the condition $[Q_2,Q_3]\in\mathfrak s$ similarly implies
that the coefficient~$c_3$ can be assumed to be equal to zero,
and we set $a_1=0$ and $a_3=1$, linearly combining~$Q_1$ and~$Q_3$.
Therefore, we also obtain the subalgebras~$\mathfrak s_{3.2}^\nu$, $\mathfrak s_{3.3}$ and~$\mathfrak s_{4.2}$.

\medskip\par\noindent	
$\boldsymbol{\pi_{\mathfrak f}\mathfrak s=\mathfrak f.}$
Hence $\hat n=3$ and the algebra~$\mathfrak s$ is nonsolvable, i.e., it has a nonzero Levi factor,
which is necessarily a Levi factor of~$\mathfrak g^{\rm ess}$ as well.
According to the Levi--Malcev theorem, we can assume modulo the $R$-equivalence
that $(Q_1,Q_2,Q_3)=(\mathcal P^t,\mathcal D,\mathcal K)$.
If $n>3$ and $(a_4,b_4)\ne(0,0)$, then
$[\mathcal P^t+\mathcal K,Q_4]=b_4\mathcal P^x-a_4\mathcal G^x\in\mathfrak s$ and
$[Q_4,b_4\mathcal P^x-a_4\mathcal G^x]=(a_4^{\,2}+b_4^{\,2})\mathcal I\in\mathfrak s$,
which means that $n=6$ and $\mathfrak s=\mathfrak g^{\rm ess}$.
This is why $G^{\rm ess}$-inequivalent proper subalgebras of $\mathfrak g^{\rm ess}$ with $\hat n=3$
are exhausted by~$\mathfrak s_{3.1}$ and~$\mathfrak s_{4.1}$.

\medskip\par	
The inequivalence, to each other, of the subalgebras~$\mathfrak s_{1.1}$ and~$\mathfrak s_{1.2}^\delta$
and the inequivalence of subalgebras within all the parameterized families
that are listed in the lemma's statement can be checked using the direct arguments.
For other pairs of subalgebras, it suffices to compare the corresponding values of
$n:=\dim\mathfrak s$, $\hat n:=\pi_{\mathfrak f}\mathfrak s$ and~$\check n':=\dim\mathfrak s\cap\mathfrak r'$.
\end{proof}

Let us show that to classify Lie reductions of the equation~\eqref{eq:LinHeat} to ordinary differential equations,
it in fact suffices to classify one-dimensional subalgebras of~$\mathfrak g^{\rm ess}$ rather than the entire algebra~$\mathfrak g$.
For this purpose, we construct a complete list of $G$-inequivalent one-dimensional subalgebras of $\mathfrak g$.
Recall that the pseudogroup~$G$ splits over its normal pseudosubgroup~$G^{\rm lin}$, $G=G^{\rm ess}\ltimes G^{\rm lin}$
(see Section~\ref{sec:LinHeatPointSymGroup}).
In other words, any element~$\Phi$ of~$G$ can be uniquely decomposed as $\Phi=\mathscr F\circ\mathscr Z(f)$
for some $\mathscr F\in G^{\rm ess}$ and some $\mathscr Z(f)\in G^{\rm lin}$.
Accordingly, the algebra $\mathfrak g$ splits over its ideal~$\mathfrak g^{\rm lin}$,
$\mathfrak g=\mathfrak g^{\rm ess}\lsemioplus\mathfrak g^{\rm lin}$.
Thus, the complete description of the adjoint action of~$G$ on~$\mathfrak g$
is given by the already described action of $G^{\rm ess}$ on $\mathfrak g^{\rm ess}$
that is supplemented with the descriptions of the adjoint actions
of~$G^{\rm lin}$ on~$\mathfrak g^{\rm ess}$ and of~$G^{\rm ess}$ on~$\mathfrak g^{\rm lin}$,
which are $\mathscr Z(f)_*Q=Q+Q[f]\p_u$ and $\mathscr F_*\mathcal Z(f)=\mathcal Z(\mathscr F_*f)$,
whereas the adjoint action of~$G^{\rm lin}$ on~$\mathfrak g^{\rm lin}$ is trivial.
Here $Q$ is an arbitrary vector field from~$\mathfrak g^{\rm ess}$,
$Q[f]$ for a function~$f$ of~$(t,x)$ denotes the evaluation of the characteristic~$Q[u]$ of~$Q$ at $u=f$,
and $\mathscr F$ is an arbitrary transformation from~$G^{\rm ess}$.

In the course of the classification, the following lemma will be of significant use.

\begin{lemma}\label{lem:ExistSolut}
Let $Q=\xi^1\p_{x_1}+\xi^2\p_{x_2}+\eta^1u\p_u$ be a Lie-symmetry vector field of
a linear partial differential equation~$\mathcal L$: $Lu=0$ in two independent variables $(x_1,x_2)$
with a linear differential operator~$L$,
where $\xi^1$, $\xi^2$ and~$\eta$ are smooth functions of $(x_1,x_2)$ with $(\xi^1,\xi^2)\ne(0,0)$.
Suppose in addition that the equation~$\mathcal L$
is not a differential consequence of the equation $Q[u]:=\eta^1u-\xi^1u_{x_1}-\xi^2u_{x_2}=0$.
Then for an arbitrary solution~$f$ of the equation~$\mathcal L$,
there exists a solution~$h$ of the same equation that in addition satisfies the constraint $Q[h]=f$.
\end{lemma}

\begin{proof}
We straighten out the vector field~$Q$ to $\p_{\tilde x_1}$
by a point transformation~$\Phi$ of the form
$\tilde x_1=X^1(x_1,x_2)$, $\tilde x_2=X^2(x_1,x_2)$, $\tilde u=U^1(x_1,x_2)u$
with $(X^1_{x_1}X^2_{x_2}-X^1_{x_2}X^2_{x_1})U^1\ne0$.
The transformation~$\Phi$ preserves the linearity and homogeneity of~$\mathcal L$.
This is why we can assume without loss of generality from the very beginning
that $Q=\p_{x_1}$.

The Lie invariance of~$\mathcal L$ with respect to this~$Q$
means that, up to a nonzero multiplier of~$L$ that may be a function of~$(x_1,x_2)$,
the coefficients of the operator~$L$ do not depend on~$x_1$.
We represent this operator in the form $L=\hat L\circ\p_{x_1}+R$,
where $\hat L$ and~$R$ are linear differential operators with coefficients depending at most on~$x_2$,
and, moreover, the operator~$R$ contains at most differentiations with respect to~$x_2$
and is nonzero since otherwise the equation~$\mathcal L$ is a differential consequence of the equation $Q[u]=0$.
The constraint $Q[h]=f$ takes the form $h_{x_1}=f$,
and its general solution is $h=\int_{x_{1,0}}^{x_1}f(s,x_2)\,{\rm d}s+\varphi(x_2)$,
where $x_{1,0}$ is an appropriate fixed value of~$x_1$,
and $\varphi$ is an arbitrary sufficiently smooth function of~$x_2$.
Substituting the expression for~$h$ into the equation~$\mathcal L$ and taking into account
that $\hat L$ and~$\p_{x_1}$ commute, $\hat L\circ\p_{x_1}=\p_{x_1}\circ\hat L$, we have
\begin{gather*}
\begin{split}
Lh(x_1,x_2)
&=\hat Lf(x_1,x_2)+\int_{x_{1,0}}^{x_1}Rf(s,x_2)\,{\rm d}s+R\varphi(x_2)\\
&=\int_{x_{1,0}}^{x_1}\p_s\circ\hat L\big|_{x_1\rightsquigarrow s}^{}f(s,x_2)\,{\rm d}s+(\hat Lf)(x_{1,0},x_2)
+\int_{x_{1,0}}^{x_1}Rf(s,x_2)\,{\rm d}s+R\varphi(x_2)\\
&=\int_{x_{1,0}}^{x_1}L\big|_{x_1\rightsquigarrow s}^{}f(s,x_2)\,{\rm d}s+(\hat Lf)(x_{1,0},x_2)+R\varphi(x_2)
 =(\hat Lf)(x_{1,0},x_2)+R\varphi(x_2).
\end{split}
\end{gather*}
As a result, the question of the existence of~$h$ reduces to solving the linear
(either ordinary differential or algebraic in the sense that is not differential) equation
$R\varphi=-(\hat Lf)(x_{1,0},\cdot)$ with respect to~$\varphi$,
whose solution definitely exists.
\end{proof}

\begin{lemma}\label{lem:LinHeat1DSubalgsOfEntireG}
A complete list of $G$-inequivalent one-dimensional subalgebras of $\mathfrak g$ consists of
the one-dimensional subalgebras of~$\mathfrak g^{\rm ess}$ listed in Lemma~\ref{lem:LinHeat1DSubalgs}
and the subalgebras of the form $\langle\mathcal Z(f)\rangle$,
where the function~$f$ belongs to a fixed complete set of $G^{\rm ess}$-inequivalent nonzero solutions
of the equation~\eqref{eq:LinHeat}.
\end{lemma}

\begin{proof}
The classification of one-dimensional subalgebras of~$\mathfrak g$ is based on the corresponding classification
for~$\mathfrak g^{\rm ess}$.
This is due to the fact that subalgebras $\mathfrak s_1$ and $\mathfrak s_2$ of $\mathfrak g$ are definitely $G$-inequivalent if
their natural projections on the subalgebra~$\mathfrak g^{\rm ess}$
under the vector-space decomposition $\mathfrak g=\mathfrak g^{\rm ess}\dotplus\mathfrak g^{\rm lin}$
are $G^{\rm ess}$-inequivalent.

Let a (nonzero) vector field~$Q$ be a basis element
of a one-dimensional subalgebra~$\mathfrak s$ of~$\mathfrak g$.
In view of the above decomposition, we can represent~$Q$ as $Q=\hat Q+\mathcal Z(f)$
for some $\hat Q\in\mathfrak g^{\rm ess}$ and some solution~$f$ of the equation~\eqref{eq:LinHeat}.

If $\hat Q\notin\langle\mathcal I\rangle$, then we push the vector field $Q$ forward by the transformation~$\mathscr Z(h)$,
where $h=h(t,x)$ is a common solution of the equations $h_t+h_{xx}=0$ and $\hat Q[h]+f=0$.
Since the operator $\p_t-\p_x^{\,2}$ cannot be factorized into two first-order differential operators,
the existence of such a solution follows from Lemma~\ref{lem:ExistSolut}.
For nonzero $\hat Q\in\langle\mathcal I\rangle$, we can set $\hat Q=\mathcal I$
by rescaling the entire~$Q$ and then applying~$\mathscr Z(-f)_*$ to $Q$.
Therefore, $G$-inequivalent one-dimensional subalgebras of~$\mathfrak g$,
whose projections on~$\mathfrak g^{\rm ess}$ are nonzero,
are exhausted by the subalgebras $\mathfrak s_{1.1}$--$\mathfrak s_{1.6}$ from Lemma~\ref{lem:LinHeat1DSubalgs}.

If $\hat Q=0$, then $Q=\mathcal Z(f)$ with nonzero~$f$.
Recall that the actions of the groups~$G$ and~$G^{\rm ess}$ on the algebra~$\mathfrak g^{\rm lin}$ coincide
and are equivalent to the corresponding actions on the solution set of the equation~\eqref{eq:LinHeat}.
\end{proof}

It is obvious that the subalgebras of the form $\langle\mathcal Z(f)\rangle$ are not appropriate
to be used within the framework of Lie reductions.
Therefore, the $G$-inequivalent reductions of the equation~\eqref{eq:LinHeat}
to ordinary differential equations are exhausted by those
that are associated with the one-dimensional subalgebras $\mathfrak s_{1.1}$--$\mathfrak s_{1.5}$
of~$\mathfrak g^{\rm ess}$ listed in Lemma~\ref{lem:LinHeat1DSubalgs}.

We do not consider the classification of Lie reductions of the equation~\eqref{eq:LinHeat} in the present paper
since an enhanced exhaustive list of inequivalent Lie invariant solutions of this equation
was presented in~\cite[Section A]{vane2021}, based on Examples 3.3 and 3.17 in~\cite{olve1993A}.
Up to combining the $G^{\rm ess}$-equivalence and linear superposition with each other,
these solutions exhaust the set of known exact solutions of the equation~\eqref{eq:LinHeat}
that have a closed form in terms of elementary and special functions.

\section{Pseudo-discrete symmetries}\label{sec:LinHeatPseudo-discreteSyms}

We would like to emphasize that the false ``discrete transformations'' $\mathscr J$ and $\mathscr K'$
not only just belong to the identity component~$G^{\rm ess}_{\rm id}$ of~$G^{\rm ess}$
but are, moreover, elements of the one-parameter subgroup that is generated by the vector field~$\mathcal Q^+$,
$\mathscr J=\mathscr Q^+(\pi)$ and $\mathscr K'=\mathscr Q^+(-\frac12\pi)$.
At the same time, many transformations from~$G^{\rm ess}_{\rm id}$ belong to no one-parameter subgroups of~$G^{\rm ess}$
or, in other words, $G^{\rm ess}\setminus\exp(\mathfrak g^{\rm ess})$ is a considerable part of~$G^{\rm ess}$.
To exhaustively describe $\exp(\mathfrak g^{\rm ess})$ and $G^{\rm ess}\setminus\exp(\mathfrak g^{\rm ess})$,
we use the following assertions.

\begin{lemma}\label{lem:1parSubgroupsOfHeads}
Let a Lie group~$G$ be a semidirect product of its Lie subgroup~$H$
acting on its normal Lie subgroup~$N$, $H<G$, $N\lhd G$ and $G=H\ltimes N$.
If an element~$g_0$ of~$G$ belongs to a one-parameter subgroup of~$G$,
then the multiplier $h_0\in H$ in the decomposition $g_0=h_0n_0$ with $n_0\in N$
belongs to a one-parameter subgroup of $H$.
\end{lemma}

\begin{proof}
Let a homomorphism $\chi\colon(\mathbb R,+)\to G$ define a one-parameter subgroup of~$G$
that contains~$g_0$, i.e., $\chi(\epsilon)=g_0$ for some $\epsilon\in\mathbb R$,
and let $\psi$ be the homomorphism of~$H$ into the automorphism group~${\rm Aut}(N)$ of~$N$
defined by $H\ni h\mapsto{\rm conj}_h|_N^{}\in{\rm Aut}(N)$.
Denote by $H\rightthreetimes_\psi N$ the external semidirect product of~$H$ by~$N$ relative to~$\psi$.
By the definition of internal semidirect product,
the map $j=(h,n)\mapsto hn$ of the product $H\rightthreetimes_\psi N$ onto~$G$ is a Lie group isomorphism,
and the natural projection $\pi_1$ of $H\rightthreetimes_\psi N$ onto its first component~$H$
is a Lie group homomorphism.
Hence the function $\hat\chi:=\pi_1\circ j^{-1}\circ\chi$ defines a one-parameter subgroup of~$H$,
and $\hat\chi(\epsilon)=\pi_1\circ j^{-1}\circ\chi(\epsilon)=\pi_1\circ j^{-1}g_0
=\pi_1(h_0,n_0)=h_0$.
\end{proof}

\begin{lemma}\label{lem:1parSubgroupsOfsl2R}
Up to the subgroup conjugation in $F\simeq{\rm SL}(2,\mathbb R)$
and rescaling the group parameter $\epsilon\in\mathbb R$,
one-parameter subgroups of $F\simeq{\rm SL}(2,\mathbb R)$ are exhausted by three groups whose elements respectively are
\[
\mathscr P^t(\epsilon)\sim\begin{pmatrix}1&\epsilon\\0&1\end{pmatrix},\quad
\mathscr D(\epsilon)\sim\begin{pmatrix}{\rm e}^\epsilon&0\\0&{\rm e}^{-\epsilon}\end{pmatrix},\quad
\mathscr Q^+(\epsilon)\sim\begin{pmatrix}\cos\epsilon&\sin\epsilon\\-\sin\epsilon&\cos\epsilon\end{pmatrix}.
\]
\end{lemma}

\begin{proof}
One-parameter subgroups of a Lie group~$G$ with Lie algebra~$\mathfrak g$ are equivalent
under the subgroup conjugation in~$G$
if and only if their generators span $G$-equivalent subalgebras of~$\mathfrak g$.
It is well known that a complete list of inequivalent one-dimensional subalgebras of $F\simeq{\rm SL}(2,\mathbb R)$
is exhausted by the subalgebras that are respectively spanned by the elements
\[
\mathcal P^t\sim\begin{pmatrix}0&1\\0&0\end{pmatrix},\quad
\mathcal D\sim\begin{pmatrix}1&0\\0&-1\end{pmatrix},\quad
\mathcal Q^+\sim\begin{pmatrix}0&1\\-1&0\end{pmatrix}
\]
of $\mathfrak f\simeq{\rm sl}(2,\mathbb R)$.
These elements respectively generate the one-parameter subgroups of $F\simeq{\rm SL}(2,\mathbb R)$
for lemma's statement.
\end{proof}

Hereafter we denote $E:=\mathop{\rm diag}(1,1)$.

\begin{corollary}\label{cor:1parSubgroupsOfsl2R}
A matrix $A\in{\rm SL}(2,\mathbb R)$ belongs to a one-parameter subgroup of ${\rm SL}(2,\mathbb R)$
if and only if either $\mathop{\rm tr}A>-2$ or $A=-E$.
Equivalently, a point transformation $\Phi\in F$ belongs to a one-parameter subgroup of~$F$
if and only if either $\alpha+\delta>-2$ or $\alpha=\delta=-1$ and $\beta=\gamma=0$.
\end{corollary}

\begin{proof}
The elements of each of the one-parameter subgroups from Lemma~\ref{lem:1parSubgroupsOfsl2R} satisfy this property,
which is stable under the subgroup conjugation in $F\simeq{\rm SL}(2,\mathbb R)$.
Conversely, any element of ${\rm SL}(2,\mathbb R)$ satisfying this property
reduces to one of the matrix forms presented in Lemma~\ref{lem:1parSubgroupsOfsl2R}.
\end{proof}

Of course, Lemma~\ref{lem:1parSubgroupsOfsl2R} and Corollary~\ref{cor:1parSubgroupsOfsl2R}
are well known, see, e.g., \cite{mosk1994a} for Corollary~\ref{cor:1parSubgroupsOfsl2R}.

Recall that the classification of elements of the group~${\rm SL}(2,\mathbb R)$ is based on the values of their traces.
The matrix $A\in{\rm SL}(2,\mathbb R)$ with $A\ne\pm E$ is called \textit{elliptic} if $|\mathop{\rm tr}A|<2$,
\textit{parabolic} if $|\mathop{\rm tr}A|=2$ and \textit{hyperbolic} if $|\mathop{\rm tr}A|>2$.
Therefore, Corollary~\ref{cor:1parSubgroupsOfsl2R} implies
that the elliptic elements, the hyperbolic and parabolic elements with positive traces, $E$ and~$-E$
constitute $\exp\big({\rm sl}(2,\mathbb R)\big)$.
Moreover, the multiplication by~$-E$ switches the signs of matrix traces and thus maps
the hyperbolic and parabolic parts of $\exp\big({\rm sl}(2,\mathbb R)\big)$
onto the complement of~$\exp\big({\rm sl}(2,\mathbb R)\big)$ in~${\rm SL}(2,\mathbb R)$, and vice versa.
Roughly speaking, the action by~$-E$ change the relation ``belongs to (no) one-parameter subgroup'' to its negation
for hyperbolic and parabolic elements of~${\rm SL}(2,\mathbb R)$.
In this sense, the role of~$-E$ on the level of one-parameter subgroups
is analogous to the role of a discrete element on the level of connected components
since such an element permutes the identity component with other components of the corresponding Lie group.
This is why we call $-E$ a \emph{pseudo-discrete} element of~${\rm SL}(2,\mathbb R)$.
In general, given a Lie group~$G$ with Lie algebra~$\mathfrak g$ and $\exp(\mathfrak g)\ne G_{\rm id}$,
we call an element~$g\in G$ \emph{pseudo-discrete}
if \[g\big(G_{\rm id}\setminus\exp(\mathfrak g)\big)\subseteq\exp(\mathfrak g).\]
Note that any one-parameter subgroup of~${\rm SL}(2,\mathbb R)$ that contains an elliptic element
consists of elliptic elements and the matrices~$\pm E$,
and thus the multiplication by~$-E$ preserves each of such subgroups.

The presented arguments allow us to state and prove a characterization of the elements of~$G^{\rm ess}$
that belong to the image~$\exp(\mathfrak g^{\rm ess})$ of the exponential map.

\begin{theorem}\label{thm:HeatEqCharacterizationTransFromOneParSubs}
A transformation~$\Phi$ from~$G^{\rm ess}$ belongs to a one-parameter subgroup of~$G^{\rm ess}$
if and only if its $F$-component~$\Phi_F$ in the decomposition $\Phi=\Phi_F\circ\Phi_R$
according to the splitting $G^{\rm ess}=F\ltimes R$ of~$G^{\rm ess}$ belongs to a one-parameter subgroup of $F$.
In other words, $\Phi\in\exp(\mathfrak g^{\rm ess})$ if and only if $\Phi_F\in\exp(\mathfrak f)$.
\end{theorem}

\begin{proof}
The ``only if'' part of lemma's statement directly follows from Lemma~\ref{lem:1parSubgroupsOfHeads}.
Let us prove the ``if'' part.
It is obvious that the proof can be done up to the subgroup conjugation in~$G^{\rm ess}$,
which indices the subgroup conjugation in~$F$.
In view of Lemma~\ref{lem:1parSubgroupsOfsl2R}, we then can assume that
$\Phi_F\in\{\mathscr P^t(\epsilon_1),\mathscr D(\epsilon_2),\mathscr Q^+(\epsilon_3),{\rm id}\}$
with fixed $\epsilon_1,\epsilon_2\ne0$ and fixed $\epsilon_3\in(0,2\pi)$.
If $\Phi_F={\rm id}$, then $\Phi\in R=\exp(\mathfrak r)$.
Recall that the radical~$\mathfrak r$ is a nilpotent algebra.
Further, we assume that $\Phi_F\ne{\rm id}$.
Consider the vector fields $Q$ of the form $Q=\hat Q+a\mathcal G^x+b\mathcal P^x+c\mathcal I$,
where $\hat Q\in\{\mathcal P^t,\mathcal D,\mathcal Q^+\}$.
Transformations from the one-parameter subgroups corresponding to these vector fields are respectively of the form
\begin{gather*}
\hat Q=\mathcal P^t\colon\quad
\tilde t=t+\epsilon_1,\quad
\tilde x=x+a\epsilon_1 t+b\epsilon_1+\tfrac12a\epsilon_1^2,
\\
\hphantom{\hat Q=\mathcal P^t\colon\quad}
\tilde u=\exp\big(-\tfrac12a\epsilon_1 x-\tfrac14a^2\epsilon_1^2t-\tfrac1{12}a^2\epsilon_1^3-\tfrac14ab\epsilon_1^2+c\epsilon_1\big)\,u,
\\[1.5ex]
\hat Q=\mathcal D\hphantom{^t}\colon\quad
\tilde t={\rm e}^{2\epsilon_2}t,\quad
\tilde x={\rm e}^{\epsilon_2}\big(x+a({\rm e}^{\epsilon_2}-1)t+b(1-{\rm e}^{-\epsilon_2})\big),
\\
\hphantom{\hat Q=\mathcal D^t\colon\quad}
\tilde u=
\exp\big(-\tfrac12a({\rm e}^{\epsilon_2}-1)x-\tfrac14a^2({\rm e}^{\epsilon_2}-1)^2t-\tfrac12ab({\rm e}^{\epsilon_2}-1-\epsilon_2)+c\epsilon_2-\tfrac12\epsilon_2\big)\,u,
\\[1.5ex]
\hat Q=\mathcal Q^+\colon\quad\!
\tilde t=\frac{\sin\epsilon_3+t\cos\epsilon_3}{\cos\epsilon_3-t\sin\epsilon_3},\quad
\tilde x=\frac{x+(a\sin\epsilon_3+b\cos\epsilon_3-b)t-a\cos\epsilon_3+a+b\sin\epsilon_3}{\cos\epsilon_3-t\sin\epsilon_3},
\\[1ex]
\hphantom{\hat Q=\mathcal Q^+\colon\quad\!}
\tilde u=
\sqrt{|\cos\epsilon_3-t\sin\epsilon_3|}\,\exp\big(\tfrac14(a^2+b^2)\epsilon_3+c\epsilon_3\big)\,u
\\[0.5ex]
\hphantom{\hat Q=\mathcal Q^+\colon\quad\!\tilde u=}
\times\exp\frac{\big((x+a)^2+b^2\big)\sin\epsilon_3+2b(1-\cos\epsilon_3)(x-bt+a)}{-4(\cos\epsilon_3-t\sin\epsilon_3)}.
\end{gather*}
In the notation of Theorem~\ref{thm:HeatEqSymGroup},
for arbitrary parameters $\lambda_1$, $\lambda_0$ and $\sigma>0$,
the transformation~$\Phi$ with reduced
$\Phi_F\in\{\mathscr P^t(\epsilon_1),\mathscr D(\epsilon_2),\mathscr Q^+(\epsilon_3)
\mid \epsilon_1,\epsilon_2\ne0,\, \epsilon_3\in(0,2\pi)\}$
belongs to the respective one-parameter subgroup,
where the parameters~$a$, $b$ and~$c$ satisfies the following system:
\begin{gather*}
\mathscr P^t(\epsilon_1)\colon\ \
\epsilon_1 a =\lambda_1,\quad
b\epsilon_1+\tfrac12a\epsilon_1^2=\lambda_0,\quad
c\epsilon_1-\tfrac1{12}a^2\epsilon_1^3-\tfrac14ab\epsilon_1^2=\ln\sigma,
\\[1ex]
\mathscr D(\epsilon_2)\colon\ \
a({\rm e}^{\epsilon_2}-1)   =\lambda_1,\quad
b(1-{\rm e}^{-\epsilon_2})=\lambda_0,\quad
c\epsilon_2-\tfrac12ab({\rm e}^{\epsilon_2}-1-\epsilon_2)-\tfrac12\epsilon_2=\ln\sigma,
\\[1ex]
\mathscr Q^+(\epsilon_3)\colon\
a\sin\epsilon_3-b(1-\cos\epsilon_3) =\lambda_1,\quad
a(1-\cos\epsilon_3)+b\sin\epsilon_3 =\lambda_0,\\
\hphantom{\mathscr Q^+(\epsilon_3)\colon\ }
c\epsilon_3+\tfrac18(a^2\!-b^2)\sin(2\epsilon_3)+\tfrac14(a^2\!+b^2)\epsilon_3-\tfrac12a^2\sin\epsilon_3
-\tfrac12ab(1\!-\!\cos\epsilon_3)\cos\epsilon_3=\ln\sigma.
\end{gather*}
It is obvious that each of the above systems has a unique solution.
\end{proof}

Recalling that $\varrho_1(\mathscr J)=-E$,
where $\varrho_1$ is an isomorphism between~$F<G^{\rm ess}$ and ${\rm SL}(2,\mathbb R)$ (see Section~\ref{sec:LinHeatPointSymGroup}),
we have the following assertion.

\begin{corollary}\label{cor:LinHeatJisPseudoDiscrete}
The transformation~$\mathscr J=(t,x,u)\mapsto(t,-x,u)$ is a pseudo-discrete element of~$G^{\rm ess}$.
\end{corollary}

Given a system of differential equations, it is natural to call
pseudo-discrete elements of its point symmetry (pseudo)group \emph{pseudo-discrete point symmetries} of this system.
In the suggested terminology, the transformation~$\mathscr J=(t,x,u)\mapsto(t,-x,u)$ can be called
a pseudo-discrete point symmetry of the equation~\eqref{eq:LinHeat}
only if the following conjecture holds true.

\begin{conjecture}\label{conj:HeatEqCharacterizationTransFromOneParSubsG}
A transformation~$\Phi$ from the pseudogroup~$G$ belongs to a one-parameter subgroup of~$G$
if and only if its $G^{\rm ess}$-component~$\Phi^{\rm ess}$ in the decomposition $\Phi=\Phi^{\rm ess}\circ\Phi^{\rm lin}$
according to the splitting $G=G^{\rm ess}\ltimes G^{\rm lin}$ of~$G$ belongs to a one-parameter subgroup of $G^{\rm ess}$.
In other words, $\Phi\in\exp(\mathfrak g)$ if and only if $\Phi^{\rm ess}\in\exp(\mathfrak g^{\rm ess})$.
\end{conjecture}

\section{Generalized symmetries}\label{sec:LinHeatGenSyms}

Since the equation~\eqref{eq:LinHeat} is an evolution equation,
we can naturally identify the quotient algebra of generalized symmetries of this equation
with respect to the equivalence of generalized symmetries
with the algebra
\[
\Sigma:=\big\{\eta[u]\p_u\mid(\mathrm D_t-\mathrm D_x^{\,\,2})\eta[u]=0\big\}
\]
of canonical representatives of equivalence classes, see~\cite[Section~5.1]{olve1993A}.
Here and in what follows
the jet variable~$u_k$ is associated with the derivative $\p^ku/\p x^k$, $k\in\mathbb N_0$,
and the jet variables $(t,x,u_k,k\in\mathbb N_0)$
constitute the standard coordinates on the manifold defined by the equation~\eqref{eq:LinHeat}
and its differential consequences in the infinite-order jet space $\mathrm J^\infty(\mathbb R^2_{t,x}\times\mathbb R_u)$
with the independent variables $(t,x)$ and the dependent variable~$u$.
The notation $\eta[u]$ stays for a differential function of~$u$
that depends on $t$, $x$ and a finite number of $u_k$.
$\mathrm D_t$ and $\mathrm D_x$ are the operators of total derivatives in~$t$ and~$x$, respectively,
that are restricted to such differential functions and the solution set of the equation~\eqref{eq:LinHeat},
\[
\mathrm D_t:=\p_t+\sum_{k=0}^\infty u_{k+2}\p_{u_k},\quad
\mathrm D_x:=\p_x+\sum_{k=0}^\infty u_{k+1}\p_{u_k}.
\]

The subspace
$\Sigma^n:=\big\{\eta[u]\p_u\in\Sigma\mid\ord\eta[u]\leqslant n\big\}$, $n\in\mathbb N_0\cup\{-\infty\}$,
of~$\Sigma$ is interpreted as the space of generalized symmetries of orders less than or equal to~$n$.%
\footnote{
The order~$\ord F[u]$ of a differential function~$F[u]$ is the highest order of derivatives of~$u$ involved in~$F[u]$
if there are such derivatives, and $\ord F[u]=-\infty$ otherwise.
If $Q=\eta[u]\p_u$, then $\ord Q:=\ord\eta[u]$.
}
The subspace~$\Sigma^{-\infty}$ can be identified with the subalgebra~$\mathfrak g^{\rm lin}$
of Lie symmetries of the equation~\eqref{eq:LinHeat}
that are associated with the linear superposition of solutions of this equation,
\[\Sigma^{-\infty}=\{\mathfrak Z(h):=h(t,x)\p_u\mid h_t=h_{xx}\}\simeq\mathfrak g^{\rm lin}.\]
The subspace family $\{\Sigma^n\mid n\in\mathbb N_0\cup\{-\infty\}\}$ filters the algebra~$\Sigma$.
Denote $\Sigma^{[n]}:=\Sigma^n/\Sigma^{n-1}$, $n\in\mathbb N$,
$\Sigma^{[0]}:=\Sigma^0/\Sigma^{-\infty}$ and $\Sigma^{[-\infty]}:=\Sigma^{-\infty}$.
The space $\Sigma^{[n]}$ is naturally identified with the space of canonical representatives of cosets of~$\Sigma^{n-1}$
and thus assumed as the space of $n$th order generalized symmetries of the equation~\eqref{eq:LinHeat},
$n\in\mathbb N_0\cup\{-\infty\}$.

In view of the linearity of the equation~\eqref{eq:LinHeat},
an important subalgebra of its generalized symmetries consists of the linear generalized symmetries,
\[
\Lambda:=\bigg\{\eta[u]\p_u\in\Sigma\ \Big|\ \exists\,n\in\mathbb N_0, \exists\,\eta^k=\eta^k(t,x), k=0,\dots,n\colon
\eta[u]=\sum_{k=0}^n\eta^k(t,x)u_k\bigg\}.
\]
The subspace $\Lambda^n:=\Lambda\cap\Sigma^n$ of~$\Lambda$ with $n\in\mathbb N_0$
is constituted by the generalized symmetries with characteristics of the form
\begin{gather}\label{eq:LinHeatLinGenSyms}
\eta[u]=\sum_{k=0}^n\eta^k(t,x)u_k.
\end{gather}
These symmetries are of order~$n$ if and only if the coefficient~$\eta^n$ does not vanish.
The quotient spaces $\Lambda^{[n]}=\Lambda^n/\Lambda^{n-1}$, $n\in\mathbb N$, and the subspace $\Lambda^{[0]}=\Lambda^0$
are naturally embedded into the respective spaces $\Sigma^{[n]}$, $n\in\mathbb N_0$,
when taking linear canonical representatives for cosets of~$\Sigma^{n-1}$ containing linear generalized symmetries.
We interpret the space $\Lambda^{[n]}$ as the space of $n$th order linear generalized symmetries
of the equation~\eqref{eq:LinHeat}, $n\in\mathbb N_0$.

\begin{lemma}\label{lem:LinHeatGenSymsDim}
$\dim\Lambda^{[n]}=n+1$, $n\in\mathbb N_0$.
\end{lemma}

\begin{proof}
The criterion of invariance of the equation~\eqref{eq:LinHeat}, $(\mathrm D_t-\mathrm D_x^{\,\,2})\eta=0$,
with respect to linear generalized symmetries with characteristics~$\eta$ of the form~\eqref{eq:LinHeatLinGenSyms}
implies the system of determining equations for the coefficients of these characteristics,
\[
\Delta_k\colon\ \eta^k_t-\eta^k_{xx}=2\eta^{k-1}_x,\quad k=0,\dots,n+1,
\]
where we assume $\eta^{-1}$ and $\eta^{n+1}$ to vanish.
We first integrate the equation~$\Delta_{n+1}$: $\eta^n_x=0$, obtaining $\eta^n=\theta(t)$ for some smooth function~$\theta$ of~$t$.
After substituting the obtained value of~$\eta^n$ into $\Delta_n$,
we consider the set~$\Delta_{[0,n]}$ of the equations~$\Delta_k$ with $k=0,\dots,n$
as a system of inhomogeneous linear differential equations with respect to the coefficients~$\eta^{k'}$, ${k'}=0,\dots,n-1$.
It is convenient to represent the equations~$\Delta_k$ with $k=1,\dots,n$ as
$\eta^{k-1}_x=\frac12(\eta^k_t-\eta^k_{xx})$ and integrate them with respect to~$x$ in descending order.
As a result, we derive the following expression for $\eta^{n-k}$, $k=0,\dots,n-1$:
\[
\eta^{n-k}=\frac1{2^kk!}\frac{{\rm d}^k\theta}{{\rm d}t^k}x^k+R^{n-k},
\]
where $R^{n-k}$ is a smooth function of $(t,x)$, which is a polynomial in~$x$ with $\deg_xR^{n-k}<k$.
In particular,
\[
\eta^0=\frac1{2^nn!}\frac{{\rm d}^n\theta}{{\rm d}t^n}x^n+R^0,
\]
where $R^0$ is a smooth function of $(t,x)$, which is a polynomial in~$x$ with $\deg_xR^0<n$.
In view of this, the equation~$\Delta_0$: $\eta^0_t-\eta^0_{xx}=0$ means that
$\eta^0$ is a polynomial solution of the linear heat equation, and thus ${\rm d}^n\theta/{\rm d}t^n=\const$,
i.e., $\theta$ is a polynomial with respect to~$t$ of degree at most~$n$, $\theta\in\mathbb R_n[t]$.
Moreover, for any $\theta\in\mathbb R_n[t]$ there is a solution of~$\Delta_{[0,n]}$.
In other words, the vector spaces~$\Lambda^{[n]}$ and~$\mathbb R_n[t]$ are isomorphic.
\end{proof}

\begin{corollary}
$\dim\Lambda^n=\sum\limits_{k=0}^n\dim\Lambda^{[k]}=\frac12(n+1)(n+2)<+\infty$, $n\in\mathbb N_0$.
\end{corollary}

\begin{lemma}\label{lemma:IDFM:SigmaIsomorphism}
$\Sigma^{[n]}=\tilde\Sigma^{[n]}:=\big\langle\mathfrak Q^{k,n-k}, k=0,\dots,n\big\rangle$,
where $\mathfrak Q^{kl}:=(\mathrm G^k\mathrm D_x^lu)\p_u$ and $\mathrm G:=t\mathrm D_x+\frac12x$.
\end{lemma}

\begin{proof}
In view of the Shapovalov--Shirokov theorem~\cite[Theorem~4.1]{shap1992a},
Lemma~\ref{lem:LinHeatGenSymsDim} implies that $\Sigma^{[n]}=\Lambda^{[n]}$ for $n\in\mathbb N_0$.
The differential functions $\mathrm D_xu=u_x$ and $\mathrm Gu=tu_x+\frac12xu$ are the characteristics
of the Lie symmetries~$-\mathcal P^x$ and~$-\mathcal G^x$ of the equation~\eqref{eq:LinHeat}, respectively,
and hence the operators~$\mathrm D_x$ and~$\mathrm G$ are recursion operators of this equation
\cite[Example~5.21]{olve1993A}.
Hence any operator~$\mathfrak Q$ in the universal enveloping algebra generated by these operators
is a symmetry operator of~\eqref{eq:LinHeat},
i.e., a generalized vector field~$(\mathfrak Qu)\p_u$ is a generalized symmetry of~\eqref{eq:LinHeat}.
Therefore, $\big\langle\mathfrak Q^{kl}, k,l\in\mathbb N_0\big\rangle\subseteq\Lambda\subseteq\Sigma$,
and thus $\tilde\Sigma^{[n]}\subseteq\Sigma^{[n]}$.
In addition, $\dim\tilde\Sigma^{[n]}=n+1=\dim\Sigma^{[n]}$, which means that $\Sigma^{[n]}=\tilde\Sigma^{[n]}$.
\end{proof}

Since $\Sigma=\Sigma^{-\infty}\dotplus\Sigma^{[0]}\dotplus\Sigma^{[1]}\dotplus\cdots$,
in view of Lemma~\ref{lemma:IDFM:SigmaIsomorphism} we obtain the following assertion.

\begin{theorem}\label{thm:LinHeatAlgOfGenSyms}
The algebra of generalized symmetries of the (1+1)-dimensional linear heat equation~\eqref{eq:LinHeat} is
$\Sigma=\Lambda\lsemioplus\Sigma^{-\infty}$, where
\begin{gather*}
\Lambda:=\big\langle\mathfrak Q^{kl},\,k,l\in\mathbb N_0\big\rangle,
\quad
\Sigma^{-\infty}:=\big\{\mathfrak Z(h)\big\}
\end{gather*}
with $\mathfrak Q^{kl}:=(\mathrm G^k\mathrm D_x^lu)\p_u$, $\mathrm G:=t\mathrm D_x+\frac12x$,
$\mathfrak Z(h):=h(t,x)\p_u$,
and the parameter function~$h$ runs through the solution set of~\eqref{eq:LinHeat}.
\end{theorem}

Thus, the algebra~$\Lambda$ of the linear generalized symmetries of the equation~\eqref{eq:LinHeat}
is generated by the two recursion operators~$\mathrm D_x$ and~$\mathrm G$ from the simplest linear generalized symmetry~$u\p_u$,
and both the recursion operators and the seed symmetry are related to Lie symmetries.
In particular, on the solution set of the equation~\eqref{eq:LinHeat},
the generalized symmetries associated with the basis elements
$\mathcal P^t$, $\mathcal D$, $\mathcal K$, $\mathcal G^x$, $\mathcal P^x$ and $\mathcal I$
of~$\mathfrak g^{\rm ess}$ are, up to sign, $\mathfrak Q^{02}$, $2\mathfrak Q^{11}+\frac12\mathfrak Q^{00}$,
$\mathfrak Q^{20}$, $\mathfrak Q^{10}$, $\mathfrak Q^{01}$, $\mathfrak Q^{00}$, respectively,
see \cite[Example~5.21]{olve1993A}.
By analogy with~$\mathfrak g^{\rm ess}$, we can call~$\Lambda$
the essential algebra of generalized symmetries of the equation~\eqref{eq:LinHeat}.

Since $\mathrm D_x\mathrm G=\mathrm G\mathrm D_x+\frac12$,
the commutation relations between the generalized vector fields spanning the algebra~$\Sigma$ are the following:
\begin{gather*}
[\mathfrak Q^{kl},\mathfrak Q^{k'l'}]
=\big((\mathrm G^{k'}\mathrm D_x^{l'}\mathrm G^k\mathrm D_x^l-\mathrm G^k\mathrm D_x^l\mathrm G^{k'}\mathrm D_x^{l'})u\big)\p_u\\[1ex]
=\sum_{i=0}^{\min(k,l')}\frac{i!}{2^i}\binom ki\binom{l'}i\mathfrak Q^{k+k'-i,\,l+l'-i}
-\sum_{i=0}^{\min(k',l)}\frac{i!}{2^i}\binom{k'}i\binom li\mathfrak Q^{k+k'-i,\,l+l'-i},
\\[1.5ex]
[\mathfrak Z(h),\mathfrak Q^{kl}]=\mathfrak Z(\mathrm G^k\mathrm D_x^lh), \quad
[\mathfrak Z(h^1),\mathfrak Z(h^2)]=0.
\end{gather*}

\noprint{
\begin{gather*}
\mathrm D_x^l\mathrm G=\mathrm G\mathrm D_x^l+\tfrac12l\mathrm D_x^{l-1},\quad
\mathrm D_x\mathrm G^k=\mathrm G^k\mathrm D_x+\tfrac12k\mathrm G^{k-1},
\\
\mathrm D_x^l\mathrm G^k=\sum_{i=0}^{\min(k,l)}\frac{i!}{2^i}\binom li\binom ki\mathrm G^{k-i}\mathrm D_x^{l-i}
\end{gather*}
}

The radical~$\mathfrak r$ of~$\mathfrak g^{\rm ess}$ is isomorphic to
the algebra~$\Lambda^1\simeq\Sigma^1/\Sigma^{-\infty}$.
Let~$\phi\colon\mathfrak r\to{\rm h}(1,\mathbb R)$
be the isomorphism with $\phi(\mathcal I)=c$, $\phi(2\mathcal G^x)=p$, $\phi(\mathcal P^x)=q$.
Up to the antisymmetry of the Lie bracket,
the only nonzero commutation relation of ${\rm h}(1,\mathbb R)$ is $[p,q]=c$.
Thus, the universal enveloping algebra~$\mathfrak U(\mathfrak r)$ of the algebra~$\mathfrak r$ is isomorphic to
the universal enveloping algebra~$\mathfrak U\big({\rm h}(1,\mathbb R)\big)$ of the algebra~${\rm h}(1,\mathbb R)$,
which is the quotient of the tensor algebra~$\mathrm T\big({\rm h}(1,\mathbb R)\big)$
by the two-sided ideal~$I$ generated by the elements
$c\otimes p-p\otimes c$, $c\otimes q-q\otimes c$, $p\otimes q-q\otimes p-c$.

Consider the associative algebra $\Upsilon_{\mathfrak r}$ generated by the differential operators $\mathrm D_x$ and $\mathrm G$.
This algebra admits the following presentation by generators and relations:
\begin{gather}\label{eq:UpsilonRPresentation}
\Upsilon_{\mathfrak r}=\big\langle
\mathrm D_x,\mathrm G\mid
[\mathrm D_x,\mathrm G]=\tfrac12
\big\rangle.
\end{gather}
The application of Bergman's diamond lemma~\cite{berg1978a}
to the algebra~$\Upsilon_{\mathfrak r}$ with the presentation~\eqref{eq:UpsilonRPresentation}
leads to a description of the basis of~$\Upsilon_{\mathfrak r}$.

\begin{lemma}\label{lem:BasisUpsilonR}
Fixed any ordering $(\mathrm Q^1,\mathrm Q^2)$ of $\{\mathrm D_x,\mathrm G\}$,
$\mathrm Q^1<\mathrm Q^2$,
the monomials of the form $\mathbf Q^\alpha:=(\mathrm Q^1)^{\alpha_1}(\mathrm Q^2)^{\alpha_2}$
with $\alpha=(\alpha_1,\alpha_2)\in\mathbb N_0^{\,\,2}$
constitute a basis of the algebra~$\Upsilon_{\mathfrak r}$.
\end{lemma}

Hereafter, we assume the ordering $\mathrm G<\mathrm D_x$.
In view of the universal property of the universal enveloping algebra $\mathfrak U(\mathfrak r)$ of $\mathfrak r$
and the Poincar\'e--Birkhoff--Witt theorem, we obtain the following useful lemma.

\begin{lemma}\label{lem:IsoUnivEnv}
In the sense of unital algebras, the algebra $\Upsilon_{\mathfrak r}$ is isomorphic
to the quotient algebra of the universal enveloping algebra $\mathfrak U(\mathfrak r)$ of~$\mathfrak r$
by the two-sided ideal $(\iota(\mathcal I)+1)$ generated by $\iota(\mathcal I)+1$,
$\Upsilon_{\mathfrak r}\simeq\mathfrak U(\mathfrak r)/(\iota(\mathcal I)+1)$,
where $\iota\colon\mathfrak r\hookrightarrow\mathfrak U(\mathfrak r)$
is the canonical embedding of the Lie algebra $\mathfrak r$ 
in its universal enveloping algebra~$\mathfrak U(\mathfrak r)$.
Moreover, this gives rise to an isomorphism between 
the associated Lie algebras~\smash{$\Upsilon_{\mathfrak r}^{(-)}$}
and \smash{$\big(\mathfrak U(\mathfrak r)/(\iota(\mathcal I)+1)\big)^{(-)}$}.
\end{lemma}

\begin{proof}
The linear extension of the correspondence $\mathcal P^x\mapsto\mathrm D_x$,
$\mathcal G^x\mapsto\mathrm G$ and $\mathcal I\mapsto-1$
defines a Lie algebra homomorphism $\varphi$ from $\mathfrak r$ to the Lie algebra $\Upsilon_{\mathfrak r}^{(-)}$
arising from the associative algebra~$\Upsilon_{\mathfrak r}$.
The universal property of the universal enveloping algebra $\mathfrak U(\mathfrak r)$
extends this homomorphism to the unital homomorphism
$\hat\varphi\colon\mathfrak U(\mathfrak r)\to\Upsilon_{\mathfrak r}$
of the associative algebras satisfying the property $\varphi=\hat\varphi\circ\iota$ as homomorphism of vector spaces.
Moreover, the homomorphism~$\hat\varphi$ is surjective since the algebra~$\Upsilon_{\mathfrak r}$
is generated by the image of~$\varphi$.

By abuse of notations, we identify $\mathfrak r$ with its image under the map $\iota$ in $\mathfrak U(\mathfrak r)$
for the remaining part of this proof.
The inclusion $(\mathcal I+1)\subset\ker\hat\varphi$ is obvious.
To demonstrate the reverse inclusion,
consider an arbitrary element $Q\in\mathfrak U(\mathfrak r)$.
According to the Poincar\'e--Birkhoff--Witt theorem, it can be expressed in the form
\[
Q=c_{i_1i_0j}(\mathcal P^x)^{i_1}(\mathcal G^x)^{i_0}\mathcal I^j
\]
with a finite number of nonzero coefficients $c_{i_1i_0j}$.
Here and in what follows we assume summation with respect to repeated indices.
Suppose that $Q\in\ker\hat\varphi$, i.e.,
\begin{gather*}
\hat\varphi(Q)=(-1)^jc_{i_1i_0j}(\mathrm D_x)^{i_1}(\mathrm G)^{i_0}=0.
\end{gather*}
Lemma~\ref{lem:BasisUpsilonR} implies that $(-1)^jc_{i_1i_0j}=0$ for each fixed tuple $(i_1,i_0)$.
Therefore,
\begin{gather*}
Q=c_{i_1i_0j}(\mathcal P^x)^{i_1}(\mathcal G^x)^{i_0}\mathcal I^j
-(-1)^jc_{i_1i_0j}(\mathcal P^x)^{i_1}(\mathcal G^x)^{i_0}
=c_{i_1i_0j}(\mathcal P^x)^{i_1}(\mathcal G^x)^{i_0}(\mathcal I^j-(-1)^j).
\end{gather*}
For each~$j$, the factor $\mathcal I^j-(-1)^j$ is divisible by $\mathcal I+1$.
Therefore, $\ker\hat\varphi=(\mathcal I+1)$
and the first isomorphism theorem for associative algebras
implies the isomorphism $\Upsilon_{\mathfrak r}\simeq\mathfrak U(\mathfrak r)/(\mathcal I+1)$.

The isomorphism between the associated Lie algebras
\smash{$\Upsilon_{\mathfrak r}^{(-)}$} and \smash{$\big(\mathfrak U(\mathfrak r)/(\mathcal I+1)\big)^{(-)}$}
follows from the definition of the Lie bracket on these algebras.%
\footnote{%
Given an associative algebra~$A$, the symbol $A^{(-)}$ denotes the Lie algebra associated with~$A$.
More specifically, $A$ and~$A^{(-)}$ coincide as vector spaces
and the Lie bracket on~$A^{(-)}$ is defined as the ring-theoretic commutator~on~$A$.
}
\end{proof}

In view of the definition ${\rm W}(1,\mathbb R):=\mathfrak U\big({\rm h}(1,\mathbb R)\big)/(c-1)$ 
of the rank-one Weyl algebra ${\rm W}(1,\mathbb R)$ and 
the isomorphisms $\mathfrak r\simeq{\rm h}(1,\mathbb R)$, 
$\mathfrak U(\mathfrak r)\simeq\mathfrak U({\rm h}(1,\mathbb R))$ and
${\rm W}(n,\mathbb R)^{\rm op}\simeq\mathfrak U\big({\rm h}(n,\mathbb R)\big)/(c+1)$,
Lemma~\ref{lem:IsoUnivEnv} gives the following.

\begin{corollary}
The algebra~$\Lambda$ of the linear generalized symmetries of the equation~\eqref{eq:LinHeat}
is isomorphic to the Lie algebra \smash{${\rm W}(1,\mathbb R)^{(-)}$}
associated with the rank-one Weyl algebra ${\rm W}(1,\mathbb R)$.
\end{corollary}

\begin{proof}
The vector space isomorphism $\varphi$ between~$\Lambda$ and~$\Upsilon_{\mathfrak r}$ is obvious.
It is the linear extension of the correspondence $(\mathrm G^k\mathrm D_x^lu)\p_u\mapsto\mathrm G^k\mathrm D_x^l$
to~$\Lambda$.
Given $Q,R\in\Upsilon_{\mathfrak r}$,
we have $[(Qu)\p_u,(Ru)\p_u]=([R,Q]u)\p_u$, e.g., in view of \cite[Proposition~5.15]{olve1993A}.
Hence, $\varphi([(Qu)\p_u,(Ru)\p_u])=[R,Q]$,
and thus~$\varphi$ is an anti-isomorphism from~$\Lambda$ to~\smash{$\Upsilon_{\mathfrak r}^{(-)}$}, 
which is anti-isomorphic to \smash{${\rm W}(1,\mathbb R)^{(-)}$}.
\end{proof}

\section{Absence of discrete point symmetries for Burgers equation}\label{sec:LinHeatBurgers}

The above formalism, which was first introduced in~\cite{kova2022a} and developed here
using the remarkable Kolmogorov equation $u_t+xu_y=u_{xx}$ and the heat equation $u_t=u_{xx}$, respectively,
can be applied to many other both linear and nonlinear systems of differential equations
whose point symmetry transformations have some components that are linear fractional in certain variables.
Among such systems (including single differential equations) are
the Burgers equation~\cite{poch2017a},
the nonlinear diffusion equation with power nonlinearity of power $-4/3$~\cite{ovsi1959a}, 
the Harry Dym equation \cite[Example~11.6]{hydo2000A}, \cite[Section~4]{hydo2000b}, 
the two-dimensional Burgers system~\cite{kont2019a}, 
the Chazy equations \cite[Example~11.5]{hydo2000A}, 
a family of third-order ordinary differential equations arising in the course of finding integrable cases of Abel equations~\cite{opan2022b},
and two-dimensional shallow water equations with flat bottom topography~\cite{bihl2020a}.
Since the Burgers equation is closely related to the heat equation~\eqref{eq:LinHeat},
below we present, within the framework of this formalism,
a refined representation of the point symmetry group of the Burgers equation
\begin{gather}\label{eq:Burgers}
v_t+vv_x=v_{xx},
\end{gather}
cf.~\cite{poch2013a}.
A similar interpretation of the point symmetry (pseudo)groups of the mentioned and other analogous systems
will be a subject of a forthcoming paper.

The maximal Lie invariance algebra of the equation~\eqref{eq:Burgers} is
$
\mathfrak g^{\rm B}=\langle\breve{\mathcal P}^t, \breve{\mathcal D}, \breve{\mathcal K}, \breve{\mathcal P}^x, \breve{\mathcal G}^x\rangle,
$
where
\begin{gather*}
\breve{\mathcal P}^t=\p_t,\quad
\breve{\mathcal D}  =2t\p_t+x\p_x-v\p_v,\quad
\breve{\mathcal K}  =t^2\p_t+tx\p_x+(x-tv)\p_v,\\
\breve{\mathcal P}^x=\p_x,\quad
\breve{\mathcal G}^x=t\p_x+\p_v.
\end{gather*}
Up to the antisymmetry of the Lie bracket,
the nonzero commutation relations between the basis elements~$\mathfrak g^{\rm B}$ are exhausted by
\begin{gather*}
[\breve{\mathcal D},\breve{\mathcal P}^t]  =-2\breve{\mathcal P}^t,\quad
[\breve{\mathcal D},  \breve{\mathcal K}]  =2\breve{\mathcal K},\quad
[\breve{\mathcal P}^t,\breve{\mathcal K}]  =\breve{\mathcal D},\\
[\breve{\mathcal P}^t,\breve{\mathcal G^x}]=\breve{\mathcal P}^x,\quad
[\breve{\mathcal D},\breve{\mathcal G}^x]  =\breve{\mathcal G}^x,\quad
[\breve{\mathcal D},\breve{\mathcal P}^x]  =-\breve{\mathcal P}^x,\quad
[\breve{\mathcal K},\breve{\mathcal P}^x]  =-\breve{\mathcal G}^x.
\end{gather*}
The algebra~$\mathfrak g^{\rm B}$ is nonsolvable.
Its radical $\breve{\mathfrak r}=\langle\breve{\mathcal G}^x,\breve{\mathcal P}^x\rangle$,
coincides with its nilradical
and is isomorphic to the abelian two-dimensional Lie algebra $2A_1$
(see, for example, the notation in~\cite{popo2003a}).
The Levi factor $\breve{\mathfrak f}=\langle\breve{\mathcal P}^t,\breve{\mathcal D},\breve{\mathcal K}\rangle$
of~$\mathfrak g^{\rm B}$ is isomorphic to ${\rm sl}(2,\mathbb R)$.
In the Levi decomposition $\mathfrak g^{\rm B}=\breve{\mathfrak f}\lsemioplus\breve{\mathfrak r}$,
the action of~$\breve{\mathfrak f}$ on~$\breve{\mathfrak r}$
coincides, in the basis $(\breve{\mathcal G}^x,\breve{\mathcal P}^x)$ of~$\breve{\mathfrak r}$,
with the real representation $\rho_1$ of~${\rm sl}(2,\mathbb R)$.
Thus, the algebra~$\mathfrak g^{\rm B}$ is isomorphic to the algebra ${\rm sl}(2,\mathbb R)\lsemioplus_{\rho_1}2A_1$.

Using results of \cite[Section~1]{poch2013a} and \cite[Theorem 5]{poch2017a} and rearranging them
in the spirit of Theorem~\ref{thm:HeatEqSymGroup},
we obtain the enhanced representation for point symmetries of the equation~\eqref{eq:Burgers}.

\begin{theorem}\label{thm:BurgersEqSymGroup}
The point symmetry group $G^{\rm B}$ of the Burgers equation~\eqref{eq:Burgers}
consists of the point transformations of the form
\begin{gather}\label{eq:BurgersEqSymGroup}
\tilde t=\frac{\alpha t+\beta}{\gamma t+\delta},\quad
\tilde x=\frac{x+\lambda_1t+\lambda_0}{\gamma t+\delta},\quad
\tilde v=(\gamma t+\delta)v-\gamma x+\lambda_1\delta-\lambda_0\gamma,
\end{gather}
where $\alpha$, $\beta$, $\gamma$, $\delta$, $\lambda_1$ and $\lambda_0$ are arbitrary constants with $\alpha\delta-\beta\gamma=1$. 
\end{theorem}

The group~$G^{\rm B}$ contains the proper normal subgroup~$\breve R$
consisting of the transformations of the form~\eqref{eq:BurgersEqSymGroup} with $(\alpha,\beta,\gamma,\delta)=(1,0,0,1)$
and, moreover, splits over it, $G^{\rm B}=\breve F\ltimes\breve R$.
Here the subgroup~$\breve F$ of~$G^{\rm B}$ is singled out by the constraints $\lambda_1=\lambda_0=0$.
The subgroups~$\breve F$ and~$\breve R$ are isomorphic to~${\rm SL}(2,\mathbb R)$ and~$(\mathbb R^2,+)$, respectively,
where $(\mathbb R^2,+)$ is the real two-dimensional connected torsion-free abelian Lie group.
These isomorphisms are established by the correspondences
\[
\varrho_1=(\alpha,\beta,\gamma,\delta)_{\alpha\delta-\beta\gamma=1}\mapsto
\begin{pmatrix}
	\alpha & \beta \\
	\gamma  & \delta
\end{pmatrix},
\qquad
(\lambda_1,\lambda_0)\mapsto(\lambda_1,\lambda_0).
\]
The standard conjugacy action of the subgroup~$\breve F<G^{\rm B}$ on the normal subgroup~$\breve R\lhd G^{\rm B}$
is given by $(\tilde\lambda_1,\tilde\lambda_0)=(\lambda_1,\lambda_0)\varrho_1(\alpha,\beta,\gamma,\delta)$.
Summing up, the group~$G^{\rm B}$ is isomorphic to the group ${\rm SL}(2,\mathbb R)\ltimes_{\breve\varphi}(\mathbb R^2,+)$,
where $\breve\varphi\colon{\rm SL}(2,\mathbb R)\to{\rm Aut}\big((\mathbb R^2,+)\big)$ is the group antihomomorphism
defined by
$\breve\varphi(\alpha,\beta,\gamma,\delta)=(\lambda_1,\lambda_0)\mapsto(\lambda_1,\lambda_0)\varrho_1(\alpha,\beta,\gamma,\delta)$.
Thus, the group~$G^{\rm B}$ is connected.
In other words, all elements of~$G^{\rm B}$ are Lie symmetries of the equation~\eqref{eq:Burgers},
and we have the following assertion, cf.\ \cite[Section~4]{hydo2000b}.

\begin{corollary}\label{cor:BurgersEqNoDiscrSyms}
The Burgers equation admits no discrete point symmetry transformations.
\end{corollary}

In the notation of Theorem~\ref{thm:BurgersEqSymGroup},
the most general form of solutions of the equation~\eqref{eq:Burgers}
that are obtained from a given solution $v=f(t,x)$ of~\eqref{eq:Burgers} under action of~$G^{\rm B}$ is
\begin{gather*}
v=\frac1{\alpha-\gamma t}\,
f\left(\frac{\delta t-\beta}{\alpha-\gamma t},
\frac x{\alpha-\gamma t}-\lambda_1\frac{\delta t-\beta}{\alpha-\gamma t}-\lambda_0\right)
-\frac{\gamma x}{\alpha-\gamma t}+\frac{\lambda_1}{\alpha-\gamma t}.
\end{gather*}

Similarly to Theorem~\ref{thm:HeatEqCharacterizationTransFromOneParSubs},
we can prove the following assertion, which, simultaneously with Corollary~\ref{cor:1parSubgroupsOfsl2R},
exhaustively describes the elements of one-parameter subgroups of~$G^{\rm B}$.

\begin{theorem}\label{thm:BurgersEqCharacterizationTransFromOneParSubs}
A transformation~$\Phi$ from~$G^{\rm B}$ belongs to a one-parameter subgroup of~$G^{\rm B}$
if and only if its $\breve F$-component~$\Phi_{\breve F}$ in the decomposition $\Phi=\Phi_{\breve F}\circ\Phi_{\breve R}$
according to the splitting $G^{\rm B}=\breve F\ltimes\breve R$ of~$G^{\rm B}$ belongs to a one-parameter subgroup of $\breve F$.
In other words, $\Phi\in\exp(\mathfrak g^{\rm B})$ if and only if $\Phi_F\in\exp(\breve{\mathfrak f})$.
\end{theorem}

\begin{corollary}
$\breve{\mathscr J}\big(G^{\rm B}\setminus\exp(\mathfrak g^{\rm B})\big)\subseteq\exp(\mathfrak g^{\rm B})$
for the transformation~$\breve{\mathscr J}:=(t,x,v)\mapsto(t,-x,-v)$, which belongs to~$G^{\rm B}$,
i.e., this transformation is a pseudo-discrete element of~$G^{\rm B}$ or, equivalently,
is a pseudo-discrete point symmetry of the equation~\eqref{eq:Burgers}.
\end{corollary}

It is a well-known fact that the equation~\eqref{eq:Burgers} is linearized to the heat equation~\eqref{eq:LinHeat}
using the Hopf--Cole transformation $v=-2u_x/u$~\cite[p.~102]{fors1906A}.
In addition to the importance of this transformation for constructing exact solutions of the Burgers equation
in the context of applied problems, it is also significant from the point of view of group analysis.
Consider the Lie group epimorphism $\rho\colon G^{\rm ess}\to G^{\rm B}$ that is induced by the Hopf--Cole transformation,
where an arbitrary element of $G^{\rm ess}$, which is defined by~\eqref{eq:HeatEqSymGroup},
is mapped by~$\rho$ to the element of~$G^{\rm B}$ given by~\eqref{eq:BurgersEqSymGroup}
with the same values of the parameters~$\alpha$, $\beta$, $\gamma$, $\delta$, $\lambda_1$ and $\lambda_0$.
The infinitesimal counterpart of~$\rho$ is given by $\rho'(Q)=\breve Q$
for $Q\in\{\mathcal P^t,\mathcal D,\mathcal K,\mathcal G^x,\mathcal P^x\}$,
and $\rho'(\mathcal I)=0$.
Since $\ker\rho={\rm Z}(G^{\rm ess})$ and $\ker\rho'=\mathfrak z(\mathfrak g^{\rm ess})=\langle\mathcal I\rangle$,
we have that the Lie group $G^{\rm ess}$ and the Lie algebra~$\mathfrak g^{\rm ess}$
are central extensions of~$G^{\rm B}$ and~$\mathfrak g^{\rm B}$, respectively.
In other words, the following diagram
\[
\begin{tikzcd}
0
\arrow[r,hook,"\iota'"]
&\mathfrak z(\mathfrak g^{\rm ess})
\arrow[r,hook,"i'"]\arrow[d,"\exp_{{\rm Z}(G^{\rm ess})}"]
&\mathfrak g^{\rm ess}
\arrow[r,twoheadrightarrow,"\rho'"]\arrow[d,"\exp_{G^{\rm ess}}"]
&\mathfrak g^{\rm B}
\arrow[r,twoheadrightarrow,"\varsigma'"]\arrow[d,"\exp_{G^{\rm B}}"]
&0\\
{\rm id}
\arrow[r,hook,"\iota"]
&{\rm Z}(G^{\rm ess})
\arrow[r,hook,"i"]
&G^{\rm ess}
\arrow[r,twoheadrightarrow,"\rho"]
&G^{\rm B}
\arrow[r,twoheadrightarrow,"\varsigma"]
&{\rm id}
\end{tikzcd}
\]
is commutative and each of its rows is exact.
Here $\iota$ and $i$ are inclusion monomorphisms, $\varsigma$ is the trivial epimorphism,
and prime denotes the differential of the corresponding Lie group homomorphism.

\section{Conclusion}\label{sec:Conclusion}

Having revisited the state-of-the-art results on classical symmetry analysis
of the linear $(1+1)$-dimensional heat equation~\eqref{eq:LinHeat},
we found out an inconsistency in the description of the structure
of the point symmetry pseudogroup~$G$ of~\eqref{eq:LinHeat},
which in turn affected the classification of subalgebras of~$\mathfrak g^{\rm ess}$.
A correct list of $G^{\rm ess}$-inequivalent subalgebras of~$\mathfrak g^{\rm ess}$
is strongly required for a correct classification of Lie reductions of~\eqref{eq:LinHeat}.
This is why the aim of the present paper was to refine and extend
the classical results on algebraic properties of the Lie pseudogroup~$G$ and its Lie algebra~$\mathfrak g$.

\looseness=-1
The pseudogroup~$G$ is given in Theorem~\ref{thm:HeatEqSymGroup}.
To compute it, we have used the direct method,
simplifying the computation due to the fact that the equation~\eqref{eq:LinHeat} belongs to the class~$\mathcal E$
of linear evolution $(1+1)$-dimensional second-order equations~\eqref{eq:LinHeatClassE}, whose equivalence groupoid is well known.
More specifically, as stated  in Proposition~\ref{prop:EquivalenceGroupE},
the class~$\mathcal E$ is normalized in the usual sense,
i.e., its equivalence groupoid coincide with the action groupoid of the equivalence pseudogroup~$G^\sim_{\mathcal E}$
of this class, see~\cite{popo2008a} for details.
The description of admissible transformations within the class~$\mathcal E$
straightforwardly implies the principle constraints for the point symmetry transformations of~\eqref{eq:LinHeat}.
Then using these constraints we have searched for the admissible transformations of~$\mathcal E$
that preserve the heat equation, or, equivalently, constitute its vertex group,
which is actually a pseudogroup.
This vertex group can in fact be identified with the pseudogroup~$G$.

The pseudogroup~$G$ splits over its normal abelian pseudosubgroup~$G^{\rm lin}$
that is associated with the linear superposition of solutions of the equation~\eqref{eq:LinHeat},
$G=G^{\rm ess}\ltimes G^{\rm lin}$.
In Section~\ref{sec:LinHeatPointSymGroup},
we have redefined the group operation in~$G$ via extending the domains of compositions of transformations from~$G$
and have thus turned the pseudogroup~$G^{\rm ess}$ into a Lie group,
which has simplified its structure.
We have found out that the group~$G^{\rm ess}$ consists of two connected components
and the only independent discrete symmetry $\mathscr I'=(t,x,u)\mapsto(t,x,-u)$ swaps them.
The identity component~$G^{\rm ess}_{\rm id}$ of~$G^{\rm ess}$ is isomorphic to the semidirect product
of the real degree-two special linear group~${\rm SL}(2,\mathbb R)$ and the real rank-one Heisenberg group~${\rm H}(1,\mathbb R)$,
where the former acts on the latter by conjugation.

One more unexpected result, which is inspired by studying the structure of the pseudogroup $G$,
is that the transformations~$\mathscr J$ and~$\mathscr K'$ given in~\eqref{eq:LinHeatTransJK'}
belong to the one-parameter subgroup of~$G$ generated by the vector field~$\mathcal Q^+$.
However, they were for a long time considered as discrete point symmetry transformations of the equation~\eqref{eq:LinHeat}.
In view of this fact, we have refined the classification of all subalgebras of~$\mathfrak g^{\rm ess}$
or, equivalently, of the special Galilei algebra in dimension 1+1.
We have also classify one-dimensional subalgebras of the entire algebra~$\mathfrak g$,
which have allowed us to conclude
that inequivalent Lie reductions of the equation~\eqref{eq:LinHeat} to ordinary differential equations
are exhausted by those related to one-dimensional subalgebras of~$\mathfrak g^{\rm ess}$.

Another interesting question for studying was when a transformation from~$G$
belongs to a one-parameter subgroup of~$G$ or, in other words, to~$\exp(\mathfrak g)$.
We have solved this problem for elements of the finite-dimensional subgroup~$G^{\rm ess}$ of~$G$
in Theorem~\ref{thm:HeatEqCharacterizationTransFromOneParSubs}
and conjectured the analogous statement for the entire pseudogroup~$G$.
In view of these assertions, the transformation~$\mathscr J$ is somewhat peculiar.
It maps $G^{\rm ess}_{\rm id}\setminus\exp(\mathfrak g^{\rm ess})$ to $\exp(\mathfrak g^{\rm ess})$,
and thus we call~$\mathscr J$ a pseudo-discrete element of~$G^{\rm ess}$, see Corollary~\ref{cor:LinHeatJisPseudoDiscrete}.
The answer to the question whether~$\mathscr J$ is a pseudo-discrete symmetry of the equation~\eqref{eq:LinHeat}, or equivalently,
pseudo-discrete element of~$G$ requires proving Conjecture~\ref{conj:HeatEqCharacterizationTransFromOneParSubsG}.

In Theorem~\ref{thm:LinHeatAlgOfGenSyms}, we have constructed an explicit representation for
the elements of the algebra~$\Sigma$ of generalized symmetries of the equation~\eqref{eq:LinHeat}.
The obvious and well-known part of the theorem is that
the algebra~$\Sigma$ is a semidirect sum of
the algebra~$\Lambda$ of the linear generalized symmetries of~\eqref{eq:LinHeat} and
the ideal~$\Sigma^{-\infty}$ associated with the linear superposition of solutions of~\eqref{eq:LinHeat}.
We have proved that the algebra~$\Lambda$ is generated by the recursion operators~$\mathrm D_x$ and~$\mathrm G$
from the simplest generalized symmetry $u\p_u$,
and thus it is isomorphic to the Lie algebra \smash{${\rm W}(1,\mathbb R)^{(-)}$}
associated with the rank-one Weyl algebra ${\rm W}(1,\mathbb R)$.

The developed technique based on redefining the transformation composition
can be applied to any system of differential equations
such that certain components of all its point symmetries are linear fractional functions.
We have illustrated this possibility using the Burgers equation as an example.
Our choice for the example is justified by the fact that
the Burgers equation~\eqref{eq:Burgers} is linearized to the heat equation~\eqref{eq:LinHeat} by the Hopf--Cole transformation.
Thus, in Section~\ref{sec:LinHeatBurgers} we have revisited the study of
the point symmetry group~$G^{\rm B}$ of the Burgers equation.
In particular, we have proved that the group~$G^{\rm B}$ is connected or, in other words,
the Burgers equation admits no discrete point symmetries.
We also have shown that the group~$G^{\rm ess}$ is a central extension of the group~$G^{\rm B}$.
The study of other systems mentioned in the first paragraph of Section~\ref{sec:LinHeatBurgers}
will be a subject of a further research.

\section*{Acknowledgements}
The authors are grateful to Alexander Bihlo, Vyacheslav Boyko, Yevhenii Chapovskyi,
Michael Kunzinger, Sergiy Maksymenko, Dmytro Popovych and Galyna Popovych for valuable discussions.
The authors thank the reviewer for useful comments.
This research was undertaken thanks to funding from the Canada Research Chairs program,
the InnovateNL LeverageR{\&}D program and the NSERC Discovery Grant program.
This research was also supported in part by the Ministry of Education,
Youth and Sports of the Czech Republic (M\v SMT \v CR) under RVO funding for I\v C47813059.
ROP expresses his gratitude for the hospitality shown by the University of Vienna during his long staying at the university.
The authors express their deepest thanks to the Armed Forces of Ukraine and the civil Ukrainian people
for their bravery and courage in defense of peace and freedom in Europe and in the entire world from russism.

\end{document}